\documentclass[12pt,a4paper]{amsart}

\usepackage{amsthm,enumerate,verbatim,amssymb,xcolor, pgf, tikz}
\usetikzlibrary{patterns}

\usepackage{cancel,soul}

\theoremstyle{plain}
\newtheorem{thm}{Theorem}[section]

\newtheorem{lemma}[thm]{Lemma}

\newtheorem{prop}[thm]{Proposition}
\newtoks\prt

\theoremstyle{definition}

\newtheorem{remark}[thm]{Remark}

\newtheorem{definition}[thm]{Definition}

\def\eqn#1$$#2$${\begin{equation}\label#1#2\end{equation}}
\def\diam{\operatorname{diam}}
\def\haus{\mathcal H}
\def\hauso{\mathcal H^{n-2}}
\def\loc{\operatorname{loc}}
\long\def\clrred#1\endred{{\color{red}#1}}
\long\def\clrmagenta#1\endmagenta{{\color{magenta}#1}}
\long\def\clrpurple#1\endpurple{{\color{purple}#1}}
\long\def\clrblue#1\endblue{{\color{blue}#1}}
\newdimen\vintkern
\def\vint{{-}\kern-\vintkern\int}

\newtoks\by
\newtoks\paper
\newtoks\book
\newtoks\jour
\newtoks\yr
\newtoks\pages
\newtoks\vol
\newtoks\publ
\def\ota{{\hbox\vol{???}}}
\def\cLear{\by=\ota\paper=\ota\book=\ota\jour=\ota\yr=\ota
\pages=\ota\vol=\ota\publ=\ota}
\def\endpaper{\the\by, {\the\paper},
\textit{\the\jour} \textbf{\the\vol} (\the\yr), \the\pages.\cLear}
\def\endbook{\the\by, \textit{\the\book}, \the\publ.\cLear}
\def\endprep{\the\by, \textit{\the\paper}, \the\jour.\cLear}
\def\endyearprep{\the\by, \textit{\the\paper}, \the\jour, (\the\yr).\cLear}

\numberwithin{equation}{section}

\def\Deg{\operatorname{Deg}}
\def\div{\operatorname{div}}
\def\ep{\varepsilon}
\def\epsilon{\varepsilon}
\def\en{\mathbb N}
\def\er{\mathbb R}
\def\F{\mathcal{F}}
\def\ff{\varphi}

\def\opartial{}
\def\phi{\varphi}
\def\rn{\er^n}
\def\topi{\operatorname{im}_T}
\def\ue{\mathbf u}

\def\ep{\varepsilon}
\def\epsilon{\varepsilon}
\def\en{\mathbb N}
\def\er{\mathbb R}
\def\ff{\varphi}
\def\rn{\er^n}

\def\loc{{\rm loc}}

\def\det{\operatorname{det}}
\def\Det{\operatorname{Det}}

\def\deg{\operatorname{deg}}

\def\dx{\,dx}
\def\dy{\,dy}

\def\cof{\operatorname{cof}}
\def\Deg{\operatorname{Deg}}
\def\div{\operatorname{div}}
\def\ep{\varepsilon}
\def\epsilon{\varepsilon}
\def\en{\mathbb N}
\def\er{\mathbb R}
\def\F{\mathcal{F}}
\def\ff{\varphi}

\def\phi{\varphi}
\def\rn{\er^n}
\def\topi{\operatorname{im}_T}
\def\ue{\mathbf u}
\def\INV{\mathrm{(INV)}}
\def\N{\mathrm{(N)}}

\begin{document}
\title[Weak limit of homeomorphisms in $W^{1,n-1}$]{Weak limit of homeomorphisms in $W^{1,n-1}$: invertibility and lower semicontinuity of energy} 

\author{Anna Dole\v{z}alov\'a}
\address{Department of Mathematical Analysis, Charles University,
So\-ko\-lovsk\'a 83, 186~00 Prague 8, Czech Republic}
\email{\tt dolezalova@karlin.mff.cuni.cz; hencl@karlin.mff.cuni.cz}

\author{Stanislav Hencl}

\author{Anastasia Molchanova}
\address{Faculty of Mathematics, University of Vienna,
Oskar-Morgenstern-Platz 1, A-1090 Vienna, Austria}
\email{\tt anastasia.molchanova@univie.ac.at}



\thanks{The first two authors were supported  
by the grant GA\v{C}R P201/21-01976S. The first author is a Ph.D. student in the University Centre for Mathematical Modelling,
Applied Analysis and Computational Mathematics and was supported by the Charles University, project GA UK No. 480120. 
The third author was supported by the European Unions Horizon 2020 research and innovation programme under the Marie Sk\l adowska-Curie grant agreement No 847693. }

\begin{abstract}
    {Let $\Omega$, $\Omega'\subset\er^n$ be bounded domains and  
        let $f_m\colon\Omega\to\Omega'$ be a sequence of homeomorphisms with positive Jacobians $J_{f_m} >0$ a.e.\ and prescribed Dirichlet boundary data.
        Let all $f_m$ satisfy the Lusin $\N$ condition and $\sup_m \int_{\Omega}(|Df_m|^{n-1}+A(|\cof Df_m|)+\phi(J_f))<\infty$, where $A$ and $\varphi$ are positive convex functions.
        Let $f$ be a weak limit of $f_m$ in $W^{1,n-1}$. 
        Provided certain growth behaviour of $A$ and $\varphi$, we show that $f$ satisfies the $\INV$ condition of Conti and De Lellis, the Lusin $\N$ condition, and polyconvex energies are lower semicontinuous.  }
\end{abstract}

\maketitle

\section{Introduction}

In this paper, we study classes of mappings that might serve as classes of deformations in 
 Continuum Mechanics
models. 
Let $\Omega\subset\rn$ be a domain, i.e., a non-empty connected open set,
and let $f\colon\Omega\to\rn$ be a mapping with $J_f>0$ a.e.
Following the pioneering papers of Ball \cite{Ball} and Ciarlet and Ne\v{c}as \cite{CN} we ask if our mapping is in some sense injective as the physical ``non-interpenetration of the matter'' indicates that a deformation should be one-to-one. We continue our study from \cite{DHM} and suggest studying the class of weak limits of Sobolev homeomorphism. We show that under natural assumptions on energy functional these limits are also invertible a.e.~and that the energy functional is weakly lower semicontinuous which makes it a suitable class for 
 variational approach. 

Concerning invertibility we use the $\INV$ condition which was introduced for $W^{1,p}$-mappings, $p>n-1$, by 
M\"uller and Spector \cite{MS} (see also e.g.\  \cite{BHMC,HMC, HeMo11,HMO, MST,SwaZie2002,SwaZie2004,T}). Informally speaking, the $\INV$ condition means that the ball $B(x,r)$ is mapped inside the image of the sphere $f(S(a,r))$ and the complement $\Omega\setminus \overline{B(x,r)}$ is mapped outside $f(S(a,r))$ (see Preliminaries for the formal definition). 
From \cite{MS} we know that mappings in this class with $J_f>0$ a.e.\ are one-to-one a.e.\ and that this class is weakly closed. 
Moreover, any mapping in this class has 
many desirable properties, it maps disjoint balls into essentially disjoint balls, $\deg(f,B,\cdot)\in\{0,1\}$ for a.e.\ ball $B$, under an additional assumption  
its distributional Jacobian equals to the absolutely continuous part of $J_f$ plus a countable sum of positive multiples of Dirac measures (these corresponds to created cavities) and so on.  

In all results in the previous paragraph the authors assume that $f\in W^{1,p}(\Omega,\rn)$ for some $p>n-1$
(see also \cite{SciStr2022} for the Orlicz--Sobolev maps with  integrability just above $n-1$). 
However, in some real models for $n=3$ one often works with integrands containing the classical Dirichlet term $|Df|^2$ 
and thus this assumption is too strong. 
Therefore, for $n=3$, Conti and De Lellis \cite{CDL} introduced the concept of $\INV$ condition also for $W^{1,2}\cap L^{\infty}$ (see also \cite{BHMCR} and \cite{BHMCR2} for some recent work) and studied Neohookean functionals of the type
\eqn{funct1}
$$
    \int_{\Omega}\left(|Df(x)|^2+\varphi(J_f(x))\right)\dx,
$$
where $\varphi$ is convex, 
$\lim_{t\to 0+}\varphi(t)=\infty$ and $\lim_{t\to\infty}\frac{\varphi(t)}{t}=\infty$. 
They proved that mappings in the $\INV$ class that satisfy $J_f>0$ a.e.\ have nice properties like mappings in \cite{MS},
but this class is not weakly closed and hence cannot be used in variational models easily. To fix this problem we add an additional term to the energy functional and we work only with the class of weak limits of homeomorphisms. 

Let us note that homeomorphisms clearly satisfy the $\INV$ condition and so do their weak limits in $W^{1,p}$, $p>n-1$ (see \cite[Lemma 3.3]{MS}). 
Unfortunately, this is not true anymore in the limiting case of limit of $W^{1,n-1}$ homeomorphisms as shown by Conti and De Lellis \cite{CDL} (see also Bouchala, Hencl and Molchanova \cite{BHM}). Let us also note that the class of weak limits of Sobolev homeomorphisms was recently characterized in the planar case by Iwaniec and Onninen \cite{IO,IO2} and De Philippis and Pratelli \cite{DPP}. Our paper contributes to the study of this class in higher dimensions $n\geq 3$. 

In our previous result with J. Mal\'y \cite{DHM} we have shown that weak limit of homeomorphisms in $W^{1,n-1}$ satisfy the $\INV$ condition once the functional \eqref{funct1} is uniformly bounded and $\varphi(t)\geq \frac{1}{t^a}$ for some proper $a$. More precisely we have shown more general statement where we assumed the uniform integrability of some distortion function (see Theorem \ref{main2} below). The main aim of this paper is to show that under reasonable conditions we have even lower semicontinuity of the corresponding energy functional and that we can apply standard Calculus of Variarions techniques in this context. 
We need to assume the following natural growth assumptions 
\eqn{varphi}
$$
    \ff \text{ is a positive convex function on }(0,\infty)\text{ with }
    \lim_{t\to 0^+}\ff(t)=\infty,\ 
    $$
\eqn{varphi2}
$$
    \lim_{t\to \infty}\frac{\ff(t)}{t}=\infty
$$
and that there is a constant $A>0$ with
\eqn{varphi3}
$$
A^{-1}\ff(t)\le \ff(2t)\leq A\ff(t),\qquad t\in (0,\infty).
$$
We further assume that our homeomorphisms have the same Dirichlet boundary data and that they satisfy the Lusin $\N$ condition, i.e.\ that for every 
$E\subset\Omega$ with $|E|=0$ we have $|f(E)|=0$. The validity of the $\INV$ condition in the next theorem is from \cite{DHM} but the `Moreover' part is entirely new. 

\begin{thm}\label{main2} 
Let $n\geq 3$, $\Omega$, $\Omega'\subset\rn$ be Lipschitz domains
and let $\ff$  
satisfy \eqref{varphi} and \eqref{varphi3}.  
Let $f_m\in W^{1,n-1}(\Omega,\Omega')$, 
$m =0,1,2\dots$, 
be a sequence of 
homeomorphisms of 
$\overline\Omega$ onto $\overline{\Omega'}$ 
with $J_{f_m}> 0$ a.e.\ 
such that 
$$
\sup_m \int_{\Omega}\left(|D f_m(x)|^{n-1}+\ff\left(J_{f_m}(x)\right)+\left(\frac{|Df_m(x)|^n}{J_{f_m(x)}}\right)^{\frac{1}{n-1}}\right)\dx<\infty. 
$$
Assume further that $f_m=f_0$ on $\partial\Omega$ for all $m\in\en$. 
Let $f$ be a weak limit of $f_m$ in $W^{1,n-1}(\Omega,\rn)$, then $f$ satisfies the $\INV$ condition. 
    
    Moreover, under the additional assumptions \eqref{varphi2} and that all $f_m$ satisfy the Lusin $\N$ condition we obtain that our $f$ satisfies the Lusin $\N$ condition and we have lower semicontinuity of energy
    \begin{equation}
        \mathcal{G}(f):=\int_{\Omega}\left(|D f(x)|^{n-1}+\ff(J_{f}(x))+\left(\frac{|Df(x)|^n}{J_{f}(x)}\right)^{\frac{1}{n-1}}\right)\leq \liminf_{m\to\infty}\mathcal{G}(f_m). 
    \end{equation}
		Further 
    $$
        \text{ for a.e.\ }x\in\Omega\text{ we have }h(f(x))=x\text{ and for a.e.\ }y\in\Omega'\text{ we have }f(h(y))=y,
    $$ 
    where $h$ is a weak-$*$ limit of (some subsequence of) $f_m^{-1}$ in $BV(\Omega',\rn)$. 
\end{thm}

Furthermore, we have found out another set of conditions under which we can conclude $\INV$ condition and similar results about lower semicontinuity 
and we study the energy functional 
$$
    \F(f)=\int_{\Omega}\left(|D f(x)|^{n-1}+A(|\cof Df(x)|)+\ff(J_f(x))\right) \dx 
$$
where 
\begin{equation}\label{A}
    A(|\cdot|) \text{ is a positive convex function  
    }\text{ with }\lim_{t\to\infty}\frac{A(t)}{t}=\infty.
\end{equation}

\begin{thm}\label{main} 
    Let $n\geq 3$ and $\Omega$, $\Omega'\subset\rn$ be bounded domains.
    Let functions $\ff$ and $A$ satisfy \eqref{varphi} and \eqref{A}. 
    Let $f_m\in W^{1,n-1}(\Omega,\rn)$, 
    $m =0,1,2\dots$, 
    be a sequence of homeomorphisms of 
    $\overline\Omega$ onto $\overline{\Omega'}$ 
    with $J_{f_m}> 0$ a.e., such that $f_m$ satisfies the Lusin $\N$ condition and  
    \begin{equation}\label{key}
        \sup_m \F(f_m)<\infty. 
    \end{equation}
    Assume further that $f_m=f_0$ on $\partial\Omega$ for all $m\in\en$. 
    Let $f$ be a weak limit of $f_m$ in $W^{1,n-1}(\Omega,\rn)$, 
    then $f$ satisfies the $\INV$ condition. 
    
    Moreover, under the additional assumption \eqref{varphi2} our $f$ satisfies the Lusin $\N$ condition and we have lower semicontinuity of energy
    \begin{equation}\label{optimistic}
        \F(f)\leq \liminf_{m\to\infty}\F(f_m). 
    \end{equation}
		Assuming further that $|\partial \Omega'|=0$ we have 
		\eqn{haha}
    $$
        \text{ for a.e.\ }x\in\Omega\text{ we have }h(f(x))=x\text{ and for a.e.\ }y\in\Omega'\text{ we have }f(h(y))=y,
    $$ 
    where $h$ is a weak limit of (some subsequence of) $f_k^{-1}$ in $W^{1,1}(\Omega',\rn)$. 
\end{thm}

After we have finished this result we were informed that most of the results in Theorem \ref{main} have been proven already by D.~Henao and C.~Mora-Corral in \cite{HMC} using different technique 
(see \cite[Theorem 5.5, Proposition 8.4 and Theorem 8.5]{HMC}). They have the result even for $f_m$ satisfying $\INV$ condition once some functional $\mathcal{E}(f_m)$ is uniformly bounded and this is true for homeomorphisms $f_m$ that satisfy Lusin $\N$ condition. We believe that our proof is still of interest as it is more geometrical and brings new ideas and techniques to this area. 
Moreover, we do not need to assume \eqref{varphi2} for the validity of $\INV$ for $f$ and the property \eqref{haha} is new.For special cases of $A$ (a power high enough), the lower semicontinuity follows from \cite{STY}.

Let us comment on our assumptions in Theorem \ref{main}.
Each homeomorphism $f_m\colon \Omega\to\Omega'$, $f_m\in W^{1,n-1}(\Omega,\rn)$ with $J_{f_m}>0$ a.e.\ satisfies $f_m^{-1}\in W^{1,1}(\Omega',\rn)$ (see \cite{CHM}). 
Moreover (see e.g.~\cite{CHM} or \cite{DS}), we have
$$
    \int_{\Omega'}|Df_m^{-1}(y)| \dy\leq \int_{\Omega}|Df_m(x)|^{n-1}\dx
$$
and hence \eqref{key} implies that there is a subsequence of $f_m^{-1}$ which converges weak-$*$ to some $h\in BV(\Omega',\rn)$. 
Using
\eqn{eee}
$$
    \sup_m \int_{\Omega}A(|\cof Df_m(x)|)\dx<\infty 
$$
we get that $Df_m^{-1}$ are equiintegrable (see Theorem \ref{equiintegrable} below) and hence (up to a subsequence) $f_m^{-1}$ converge to $h\in W^{1,1}(\Omega',\rn)$ weakly in $W^{1,1}(\Omega',\rn)$. This assumption \eqref{eee} is also crucial in our proof of the $\INV$ condition as it implies that image $f_m(A)$ of small set $A\subset \partial B(c,r)$ is ``uniformly'' small in $m$ and therefore cannot enclose some big set that would like to escape from $f(\partial B(c,r))$ violating the $\INV$ condition. 

The condition
\[
    \sup_m \int_{\Omega}\ff(J_{f_m}(x))\dx<\infty\text{ with }\lim_{t\to 0+}\varphi(t)=\infty\ \Bigl(\text{resp.} \lim_{t\to \infty}\frac{\ff(t)}{t}=\infty\Bigr)
\]
implies that small sets have uniformly small preimages (resp.\ small sets have uniformly small images) and these conditions are quite standard in the theory. 
Moreover, we need to assume that $f_m$ maps null sets to null sets (by the Lusin $\N$ condition), which is again natural as our deformation cannot create a new material from ``nothing''. 
Let us note that this is crucial for the lower semicontinuity of our functional. 
In Lemma~\ref{lem:example} below we construct a series of homeomorphisms that do not satisfy the Lusin $\N$ condition as they map some null set to a set of positive measure $a$, though they satisfy all other assumptions, converge weakly to $f(x)=x$ and
\[
    \int_{(0,1)^n}J_{f_m}(x)\dx=1-a<1=\int_{(0,1)^n} J_f(x)\dx
\]
and hence lower semicontinuity fails at least for some polyconvex functionals. Similarly,
if we omit the condition~\eqref{varphi2}, we can construct a counterexample to semicontinuity of some functional if all $f_m$ satisfy $\N$ but $f$ does not. 
The lower semicontinuity of functionals below the natural $W^{1,n}$ energy has attracted a lot of attention in the past and we refer the reader e.g.~to Ball and Murat 
\cite{BM}, Mal\'y \cite{M}, Dal Maso and Sbordone \cite{DMS} and  Celada and Dal Maso \cite{CM} for further information. 

Let us note one disadvantage of our approach. In the previous models \cite{CDL}, \cite{MS} it was possible to model also the cavitation, i.e., the creation of small holes. Unfortunately, this is not possible for us as the condition \eqref{varphi2}
together with \eqref{key} tells us that $f_m$ cannot map small sets onto big sets. However, this is exactly what is needed to be done by our approximating homeomorphisms around the point where the cavity is created by $f$. On the other hand, the condition
 \eqref{varphi2}
is crucial for us in order to prove the Lusin $\N$ condition for $f$ and this condition is the key for the proof and the validity of the lower semicontinuity of our functional.



Let us briefly comment on the structure of this paper. We recall the definition of the degree and of the $\INV$ condition in the Preliminaries and we prove the equiintegrability of $Df_m^{-1}$ there. 
Our proof of $\INV$ condition for $f$ uses some techniques and results that we have developed in our previous paper \cite{DHM} on this topic. We recall some of those in the Preliminaries and then we give a detailed proof of the $\INV$ condition using some of those techniques in Section \ref{three}. 
In Sections \ref{sec:Lusin}--\ref{sec:injectivity} we use the $\INV$ condition to prove that $f$ satisfies the $\N$ condition and that $h$ ($W^{1,1}$ weak limit of $f_m^{-1}$) is the ``a.e.~inverse'' of~$f$. 
Then we use the $\N$ condition to prove the lower semicontinuity of our polyconvex functional in Section \ref{sec:lsc} and we show some counterexamples to lower semicontinuity without our assumption \eqref{varphi2} in Section \ref{counter}. Finally, we return to the result of \cite{DHM} where we have shown $\INV$ under different assumptions and we show that the lower semicontinuity of energy is valid also there if we additionally assume that $f_m$ satisfy $\N$ and that we have
 \eqref{varphi2}.
In the last Section \ref{calcvar} we give a quick application of our result in Calculus of Variations.


\section{Preliminaries} 

\subsection{Change of variables estimates}\label{area}\

Let $\Omega\subset\rn$ be open, $A\subset\Omega$ be measurable and let $g\in W_{\loc}^{1,1}(\Omega;\rn)$
be one-to-one. 
Without any additional assumption we have (see e.g.\ \cite[Theorem A.35]{HK} for $\eta=\chi_{g(A)}$)
\begin{equation}\label{area1}
    \int_{A}|J_g(x)| \dx\le |g(A)|.
\end{equation}
Moreover, for general $g$ satisfying the Lusin $\N$ condition we have (see e.g.~\cite[Theorem A.35]{HK} for $\eta=\chi_{g(A)}$)
\begin{equation}\label{area2}
    \int_{A}|J_g(x)| \dx= \int_{\rn}N(g,A,y) \dy,
\end{equation}
where $N(g,y,A)$ is defined as a number of preimages of $y$ under $g$ in $A$. 

Analogous change of variables formula holds also for mappings 
$h\colon\Omega\to\rn$, $\Omega\subset\er^{n-1}$. For Lipschitz $h$ we have (see e.g.~\cite[Theorem 3.2.3]{Fe})
\eqn{area3}
$$
\int_A J_{n-1} h(x)\; dx=\int_{\rn}N(h,A,y)\; d\haus^{n-1}(y),
$$
where $A\subset \Omega$ is measurable, $N(h,A,y)$ denotes the number of preimages $h^{-1}(y)$ in a set~$A$ and $J_{n-1}h$ is the $(n-1)$-dimensional Jacobian of $h$, i.e.\ it consists of sizes of $(n-1)\times(n-1)$ subdeterminants. 
We know that each $h\in W^{1,1}(\Omega,\rn)$ is approximately differentiable a.e.\ (see e.g.~\cite[Theorem 3.83]{AFP}) and for each 
approximately differentiable function we can exhaust $\Omega$ up to a set of measure
zero by sets so that the restriction of $h$ is Lipschitz continuous on those sets (see \cite[Theorem 3.1.8 and Theorem 3.1.4]{Fe}). It follows that \eqref{area3} holds for Sobolev mapping $h\in W^{1,1}(\Omega,\rn)$ if we know that for every $E\subset\Omega$ with $\haus^{n-1}(E)=0$ we have $\haus^{n-1}(h(E))=0$. In general the area formula \eqref{area3} holds for Sobolev $h$ only up to a set of $(n-1)$-dimensional measure zero $E\subset\Omega$ (see also \cite[Chapter 3, Section 1.5, Theorem 1 and Corollary 2]{GiaModSou}).

\begin{lemma}\label{l:reverse}
    Given $C_1<\infty$ and $\varphi$ satisfying \eqref{varphi}, 
    there exist monotone functions $\Phi$,~$\Psi\colon (0,\infty)\to(0,\infty)$ with 
    $$
        \lim_{s\to 0^+}\Phi(s)=0\text{ and }\lim_{s\to 0^+}\Psi(s)=0
    $$
    such that: 
    Let $g\in W^{1,1}(\Omega,\rn)$ be a one-to-one mapping with $\int_{\Omega} \phi(J_g)\le C_1$. 
    Then for each measurable set $A\subset\Omega$ we have
    \begin{equation}\label{reverse}
        \Phi(|A|)\le |g(A)|.
    \end{equation}
    If we moreover assume that the Lusin $\N$ condition holds for $g$ and that \eqref{varphi2} holds, then also
    \begin{equation}\label{reverse2}
        |g(A)|\le \Psi(|A|).
    \end{equation} 
\end{lemma}
\begin{proof}
    The proof of \eqref{reverse} can be found in the proof of \cite[Lemma 2.9]{DHM} (we omit here the assumption on $\|g\|_{L^{\infty}}$ as it is not necessary).
    The proof of \eqref{reverse2} follows from De la Vallee Pousin theorem \cite[Theorem B.103]{L} applied to $|J_g|$ and the fact that the Lusin $\N$ condition implies an equality in \eqref{area1}. Note that we can assume that both $\Phi$ and $\Psi$ are monotone.
\end{proof}

The following lemma was shown in \cite[Lemma 2.8]{DHM}. 
\begin{lemma}\label{Ninv} 
Let $\Omega\subset\rn$ be an open set of finite measure
and $f\in W_{\loc}^{1,1}(\Omega;\rn)$ satisfy $J_f\neq 0$ a.e. 
Then
for every $\ep>0$ there is $\delta>0$ such that for every measurable set $F\subset\rn$ 
we have
$$
|F|<\delta\implies |f^{-1}(F)|<\ep.
$$
\end{lemma}

In order to apply the previous lemma we use the following observation.

\begin{lemma}\label{jfnonzero}
Let $\Omega\subset\rn$ be open, and let $f_k\in W^{1,1}(\Omega,\rn)$ be a sequence of homeomorphisms with $J_{f_k}>0$ a.e.\ such that $f_k\to f\in W^{1,1}(\Omega,\rn)$ pointwise a.e. Assume further that
$$
\sup_k\int_{\Omega}\varphi(J_{f_k}(x))\dx<\infty,
$$ 
where $\varphi$ satisfies \eqref{varphi}. Then $J_f\neq 0$ a.e. 
\end{lemma}
\begin{proof}
Assume by contradiction that 
$$
E:=\{x\in\Omega:\ J_f(x)= 0\}\text{ satisfies }|E|>0. 
$$
As usual we find a set $E_0\subset E$ with $|E_0|=|E|$ such that the $\N$ condition holds on $E_0$ for $f$ (see e.g.\ \cite[proof of Theorem A.35]{HK}). Moreover, we assume that 
$f_k(x)\to f(x)$ for every $x\in E_0$. By \eqref{area2} we obtain
$$
|f(E_0)|=0. 
$$
We find an open set $G\subset\Omega$ such that
$$
f(E_0)\subset G\text{ and }|G|<\tfrac{1}{2}\Phi\bigl(\tfrac{1}{2}|E_0|\bigr), 
$$
where $\Phi$ comes from Lemma \ref{l:reverse}. Since $f_k(x)\to f(x)$ we can find $k_0(x)$ such that for every $k\geq k_0(x)$ we have $f_k(x)\in G$. It follows that 
$$
E_0=\bigcup_{k_0=1}^{\infty}E_{k_0},\text{ where }E_{k_0}=\{x\in E_0:\ f_k(x)\in G\text{ for every }k\geq k_0\}.  
$$ 
These sets are nested and hence we can fix $k_0$ such that $|E_{k_0}|>\tfrac{1}{2}|E_0|$. It follows that 
$$
f_{k_0}(E_{k_0})\subset G\text{ with }|E_{k_0}|>\tfrac{1}{2}|E_0|\text{ and }|G|<\tfrac{1}{2}\Phi\bigl(\tfrac{1}{2}|E_0|\bigr)
$$
which contradicts \eqref{reverse}. 

\end{proof}

\begin{thm}\label{chm}
Let $B(c,R)\subset\rn$ and let $g\in W^{1,n-1}(B(c,R),\rn)$ be a homeomorphism. Then for a.e.\ $r\in (0,R)$ we know that $g\in W^{1,n-1}(\partial B(c,r),\rn)$ and that $g$ satisfies the Lusin $\N$ condition on the sphere $\partial B(c,r)$, i.e.,
$$
\text{ for every }E\subset \partial B(c,r)\text{ with }\haus^{n-1}(E)=0\text{ we have }\haus^{n-1}(g(E))=0. 
$$
Moreover, for such $r$ and every relatively open set $E\subset \partial B(c,r)$ we have
\eqn{estimate}
$$
\haus^{n-1}(g(E))\leq C(r)\int_{E}|\cof Dg| \,d\haus^{n-1}.
$$
\end{thm}
\begin{proof}
The fact that $g\in W^{1,n-1}(\partial B(c,r),\rn)$ for a.e.\ $r$ is standard and follows e.g.\ by using the ACL condition (on circles and not lines). 
The part about the validity of Lusin $\N$ condition on a.e.\ sphere follows from \cite[Lemma 4.1]{CHM}. 

Let us have a homeomorphism $h\in W^{1,1}(\er^{n-1},\rn)$ which satisfies the Lusin $\N$ condition. 
Then the area formula~\eqref{area3} implies that for every measurable set $E\subset\er^{n-1}$ we have
$$
\haus^{n-1}(h(E))=\int_E|J_{n-1}h(x)| \dx,
$$
where $J_{n-1}h$ is the $(n-1)$-dimensional Jacobian, i.e.\ it consists of all $(n-1)\times(n-1)$ subdeterminants. 
To obtain the wanted estimate we simply do a bilipschitz change of variables (locally) from round $\partial B(c,r)$ to flat $\er^{n-1}$ and the result for $h$ implies our estimate \eqref{estimate} for $g$. Of course the constant in the bilipschitz change of variables might depend on $r$ so our constant in \eqref{estimate} could depend on $r$. 
\end{proof}

\subsection{Equiintegrability of $Df^{-1}_m$} 

The following theorem tells us that our mappings $f_m$ from Theorem \ref{main} have equiintegrable $Df^{-1}_m$. It follows that up to a subsequence $f^{-1}_{m}$ converge weakly to some $h\in W^{1,1}(\Omega',\rn)$, see \cite[Theorem B.103]{L} and \cite[Lemma 1.2 in Chapter 2.1]{Dac2008}. 

\begin{thm}\label{equiintegrable} 
    Let $\Omega$, $\Omega'\subset\rn$ be domains. Let functions $\ff$ and $A$ satisfy \eqref{varphi} and \eqref{A}. Then there is a continuous monotone function $B$ with $\lim_{t\to\infty}\frac{B(t)}{t}=\infty$ such that: \\ 
    Let $f_m\in W^{1,n-1}(\Omega,\Omega')$ be homeomorphisms with $J_{f_m}(x)>0$ a.e., $J_{f_m^{-1}}(y)>0$ a.e.\ and 
    $$
        \sup_m\int_{\Omega}\left(|D f_m(x)|^{n-1}+A(|\cof Df_m(x)|)+\ff(J_{f_m}(x))\right)\dx<\infty. 
    $$
    Then 
    $$
        \sup _m\int_{\Omega'}B\bigl(|Df_{m}^{-1}(y)|\bigr)\dy<\infty. 
    $$
\end{thm}

\begin{proof}
    Let us write $A(t)=ta(t)$, $B(t)=tb(t)$ and assume that $b$ is a suitable function such that $b(t)\leq a(t)$,
    \begin{equation}\label{pseudolog}
        b(st)\leq b(s)+b(t),
    \end{equation}
    and $B$ is continuous and monotone with superlinear growth. We give a detailed construction of such $b$ below. 
    By differentiation of $f_m^{-1}(f_m(x))=x$ we obtain
    $$
        Df_m^{-1}(f_m(x))Df_m(x)=I\text{ and }J_{f_m^{-1}}(f_m(x))J_{f_m}(x)=1 
    $$
    for a.e.~$x$ (see \cite[Lemma 2.1]{FusMosSbo2008}).
    Using the previous line, \eqref{pseudolog}, \eqref{area1} and $A\cdot\cof A=\det A\cdot I$ we have
    \begin{equation*}
    \begin{aligned}
        \int_{\Omega'}&B(|Df_{m}^{-1}(y)|)\dy= \int_{\Omega'}|Df_{m}^{-1}(y)|b(|Df_{m}^{-1}(y)|)\frac{J_{f_m^{-1}}(y)}{J_{f_m^{-1}}(y)}\dy\\
        &\leq \int_{\Omega}|Df_{m}^{-1}(f_m(x))|\ b(|Df_{m}^{-1}(f_m(x))|)\frac{1}{J_{f_m^{-1}}(f_m(x))}\dx\\
        &=\int_{\Omega}|(Df_{m}(x))^{-1}|\ b(|(Df_{m}(x))^{-1}|)J_{f_m(x)}\dx\\
        &=\int_{\Omega}|\cof Df_{m}(x)|\ b\left(|\cof Df_{m}(x)| \frac{1}{J_{f_m}(x)}\right)\dx\\
        &\leq \int_{\Omega}|\cof Df_{m}(x)|\ b(|\cof Df_{m}(x)|)\dx+\int_{\Omega}|\cof Df_{m}(x)|\ b\left(\frac{1}{J_{f_m}(x)}\right)\dx.
    \end{aligned}
    \end{equation*}
    From $B(t)=tb(t)\leq A(t)$ we obtain that the first term is uniformly bounded. By the Young inequality, we estimate the second term
    $$
        \int_{\Omega}|\cof Df_{m}(x)|\ b\left(\frac{1}{J_{f_m}(x)}\right)\dx\leq \int_{\Omega}A(|\cof Df_{m}(x)|)\dx+ \int_{\Omega}A'\left(b\left(\frac{1}{J_{f_m}(x)}\right)\right)\dx
    $$
    where $A'$ is the fixed conjugate function to our convex function $A$ (see \cite[Chapter 2.4]{HH}).
    If we ask also for 
        $$
        b(t)\leq (A')^{-1}(\varphi(\tfrac{1}{t}))
    $$
    for large $t$, we have 
    $$
        A'\left(b\left(\frac{1}{J_{f_m}(x)}\right)\right)\leq \ff(J_{f_m}(x)) + C
    $$
    for every $t$ and we are finished.

    Now we find such function $b$. We define auxiliary functions $\overline \psi$ and $\overline b$ and take $\psi$ and $b$ which are smaller than their counterparts and monotone continuous.
    Let us set 
    $$\overline\psi(t)=\frac{a(t)}{\log(t)}$$
     for $t>1$. It is continuous and from the continuity and positivity of $a$ on $(0, \infty)$ we know that $\lim_{t\to 1+} \overline\psi(t)=\infty$. Therefore we can define
    $$
        \psi(t)=
        \begin{cases}1,& 1\leq t< t_0=\min\{s>1: \overline\psi(s)=1\},\\
        \min\{\overline\psi(s), s\in[t_0, t]\},& t_0 \leq t,
        \end{cases}
    $$
    which is a positive continuous nonincreasing function less or equal to $\overline\psi$.
 
    Define
    $$
        \overline b(t)=
        \begin{cases}0,& 0<t\leq 1,\\
        \psi(t)\log(t),& 1<t<\infty.
        \end{cases}
    $$
    Since $\psi$ is continuous nonincreasing bounded on $(1,\infty)$, $\overline b$ is also continuous and \eqref{pseudolog} holds for $s, t\geq 1$ (since $\psi$ is nonincreasing and $\log$ satisfies \eqref{pseudolog}) and $s,t<1$ (since $\overline{b}(st)=0$). Moreover we have $\overline b(t)\leq a(t)$.
Now we check that 
    $$
        \lim_{t\to\infty} \overline b(t) = \lim_{t\to\infty} \psi(t)\log(t)=\infty.
    $$
    Either $\lim_{t\to\infty}\psi(t)>0$ and the statement holds, or $\lim_{t\to\infty}\psi(t)=0$. In the later case, we can find a sequence $t_0<t_1<t_2\dots$ such that 
    $$t_k=\min\{s>t_0:\overline\psi(s)=1/k\}$$
     and $t_k\to\infty$ (since $\overline \psi$ is positive). Then on $[t_0,t_k]$ we have $\psi(t)\geq 1/k =\overline \psi (t_k)= a(t_k)/\log(t_k)$ and consequentially 
    \begin{align*}
        \liminf_{t\to\infty} \overline b(t) &= \liminf_{k\to\infty} \min_{t\in [t_{k}, t_{k+1}]} \overline b(t) = \liminf_{k\to\infty} \min_{t\in [t_{k}, t_{k+1}]} \psi(t)\log(t)\geq \liminf_{k\to\infty}  \psi(t_{k+1})\log(t_k)\\
        &\geq \liminf_{k\to\infty} \frac{\log(t_k)}{k+1} \geq \frac{1}{2} \liminf_{k\to\infty} \frac{\log(t_k)}{k} = \frac{1}{2} \liminf_{k\to\infty} a(t_k)=\infty.
    \end{align*}    
        Now we want to resolve \eqref{pseudolog} for $s<1$, $t\geq 1$. If $st\leq 1$, it is clear. In the other case, we need $\overline b(st)\leq \overline b(t)$ which we do not have for $\overline b$ in general as it does not have to be monotonous. Therefore we define
    $$
        b(t)=\inf_{s\in [t, \infty)} \overline b(s).
    $$
    That function is clearly monotone, continuous, smaller than $\overline b$ and tends to $\infty$. For $s,t<1$ \eqref{pseudolog} is still trivial. For $s<1$, $t\geq 1$ it follows from monotonicity. For $s,t\geq 1$ we find $s_0=\max \{r: \overline b(r) = b(s)\}$ and $t_0$ analogously. Obviously from the definition of $b$ we have $s\leq s_0$, $t\leq t_0$. Then we have
    $$
        b(st)\leq b(s_0 t_0) \leq \overline b(s_0 t_0) \leq \overline b(s_0) + \overline b(t_0) = b(s) + b(t).
    $$
    Now we know that $B(t)=tb(t)$ is continuous nonnegative non-decreasing with $B(0)=0$ and $\lim_{t\to\infty}\frac{B(t)}{t}=\infty$.

     The last step is to show that we can ask 
    $$
        b(t)\leq (A')^{-1}(\varphi(\tfrac{1}{t}))
    $$
    for large $t$.
     We can use a similar procedure as before (replacing $a$ by $\min\{a, (A')^{-1}\}$), since $\lim_{t\to\infty}(A')^{-1}(t)=\infty$ ($A'$ is convex, negative in $0$ and positive for large values, so going to $\infty$ --- and so does its inverse, too).

\end{proof}

\subsection{Degree for continuous mappings}
Let $\Omega\subset\rn$ be a bounded open set.
Given a continuous map $f\colon\overline\Omega\to\rn$ and $y\in \rn\setminus f(\partial\Omega)$,
we can define the {\it topological degree} as
$$\deg(f,\Omega,y)=\sum_{\Omega\cap f^{-1}(y)} \operatorname{sgn}(J_f(x))$$
if $f$ is smooth in $\Omega$ and $J_f(x)\neq 0$ for each $x\in \Omega\cap f^{-1}(y)$.
By uniform approximation, 
this definition can be extended to an arbitrary continuous mapping 
$f\colon\overline\Omega\to\rn$. Note that the degree depends only on
values of $f$ on $\partial \Omega$.

If $f\colon \overline\Omega\to\rn$ is a homeomorphism,
then either $\deg (f,\Omega,y)=1$ for all $y\in f(\Omega)$
($f$ is \textit{sense preserving}), or 
$\deg (f,\Omega,y)=-1$ for all $y\in f(\Omega)$
($f$ is \textit{sense reversing}). If, in addition,
$f\in W^{1,n-1}(\Omega,\rn)$, then this topological orientation
corresponds to the sign of the Jacobian. More precisely, we have

\begin{prop}[\cite{HM}]\label{p:top=anal} 
    Let $f\in W^{1,n-1}(\Omega,\rn)$ be 
    a homeomorphism on $\overline\Omega$ with $J_f>0$ a.e.
    Then 
    $$
        \deg(f,\Omega,y)=1,\qquad y\in f(\Omega).
    $$
\end{prop}

\medskip

\subsection{Degree for $W^{1,n-1}\cap L^{\infty}$ mappings}\label{degree}
 
Let $B$ be a ball,  
$f\in W^{1,n-1}(\partial B,\rn)\cap C(\partial B,\rn)$, $|f(\partial B)|=0$, 
and $\ue\in C^1(\rn,\rn)$, then (see \cite[Proposition 2.1]{MS})
\begin{equation}\label{weakdegree}
    \int_{\rn}\deg(f,B,y)\operatorname{div} \ue(y)\dy=
    \int_{\partial B} (\ue\circ f)\cdot (\Lambda_{n-1} D_{\tau}f)\nu\, d\haus^{n-1},
\end{equation}
where $D_{\tau}f$ denotes the tangential gradient and $\Lambda_{n-1} D_{\tau}f$ is the restriction of $\operatorname{cof} Df$ to the corresponding subspace (see \cite{DHM} for details). 

Let $\mathcal M(\rn)=C_0(\rn)^*$ be the space of all signed Radon measures on $\rn$.
By \eqref{weakdegree} we see that $\deg(f,B,\cdot)\in BV(\rn)$ and
\begin{equation}\label{BVest}
    \|D\deg(f,B,\cdot)\|_{\mathcal M(\rn)}\le 
    C\|\Lambda_{n-1} D_{\tau}f\|_{L^1(\partial B)}\le C\|D_{\tau}f\|_{L^{n-1}(\partial B)}^{n-1}.
\end{equation}

Following \cite{CDL} (see also \cite{BN}) we need a more general version 
of the degree 
which works for mappings in $W^{1,n-1}\cap L^{\infty}$ 
that are not necessarily continuous. 
Although only the three-dimensional case 
is discussed
in \cite{CDL}, the arguments pass in the general case as well. The definition is in fact based 
on \eqref{weakdegree}.

\begin{definition}\label{defdegree}
    Let $B\subset\rn$ be a ball and let $f\in W^{1,n-1}(\partial B,\rn)\cap L^{\infty}(\partial B,\rn)$. Then we define 
    $\Deg(f, B, \cdot)$ as the distribution satisfying
    \begin{equation}\label{qqq}
        \int_{\rn}\Deg(f,B,y)\psi(y)\dy=
        \int_{\partial B} (\ue\circ f)\cdot(\Lambda_{n-1} D_{\tau}f) \nu\, d\haus^{n-1}
    \end{equation}
    for every test function $\psi\in C_c^{\infty}(\rn)$ 
    and every 
    $C^{\infty}$ vector field $\ue$ on $\rn$ satisfying $\div \ue=\psi$.
\end{definition}

As in \cite{CDL} (see also \cite{DHM}) it can be verified that the right-hand side does not depend on the way 
$\psi$ is expressed as $\div \ue$ and that the distribution $\Deg(f, B,\cdot)$ can be represented as a $BV$ function.  

Assume that $f$, $g\in W^{1,n-1}(\partial B,\rn)\cap L^{\infty}(\partial B,\rn)$. As in \cite[(2.5)]{DHM} we obtain the following version of some ``weak isoperimetric inequality''
\begin{equation}\label{odhad}
\begin{aligned}
    \bigl|\bigl\{y\in\rn:\ &\Deg(f,\opartial B,y)\neq \Deg(g,\opartial B,y)\bigr\}\bigr|^{\frac{n-1}{n}}
    \leq \\
    &\leq C \int_{\partial B\cap\{f\neq g\}}\left( |D_{\tau}f(x)|^{n-1}+|D_{\tau}g(x)|^{n-1}\right)d\haus^{n-1}(x).\\
\end{aligned}
\end{equation}

We need also the classical isoperimetric inequality (see e.g.\ \cite[Theorem 2 in section 5.6.2 and Theorem 2 in section 5.7.3]{EG}). Let $E\subset\rn$ be and open set with finite perimeter. 
Then 
\eqn{isoperimetric}
$$
|E|^{1-\frac{1}{n}}\leq C \haus^{n-1}(\partial E). 
$$

\begin{remark}\label{Deg=deg}
    Let $B$ be a ball and $f\in W^{1,n-1}(\partial B,\rn)\cap C(\overline B,\rn)$.
    If $|f(\partial B)|=0$, then $\Deg(f,B,y)=\deg(f,B,y)$ for 
    a.e.\  $y\in\rn$. We use different symbols to distinguish and emphasize that
    $\deg$ is defined pointwise on $\rn\setminus f(\partial B)$, whereas
    $\Deg$ is determined only up to a set of measure zero.
\end{remark}

\subsection{$\INV$ condition}
Analogously  to \cite{CDL} (see also \cite{MS}) we define the $\INV$ class.

Let $A\subset\Omega\subset\rn$. We say that $x\in \rn$ is a \textit{point of density one} (or just \textit{point of density}) of a set $A$ if
$$
\lim_{r\to 0+}\frac{|B(x,r)\cap A|}{|B(x,r)|}=1. 
$$
It is well-known that a.e.~$x\in A$ is a point of density of $A$. 

\begin{definition}[geometrical image]
    Let $\Omega\subset\rn$ be open, $f\colon \Omega \to \rn$ be a function which is approximately differentiable almost everywhere. Given a set $A \subset \Omega$ we call the geometrical image of $A$ through $f$ the set given by $f\left(\Omega_{d} \cap A\right)$, where $\Omega_d$ denotes the set where $f$ is approximatively differentiable. 
    Further on, we denote this geometrical image by $f(A)$ (since $f$ is nevertheless defined only up to a set of measure zero).
\end{definition}

\begin{definition}[topological image]\label{image} 
    Let $B\subset\rn$ be a ball and let $f\in W^{1,n-1}(\partial B,\rn)\cap L^{\infty}(\partial B,\rn)$. 
    We define the topological image of $B$ under $f$, $\topi(f,B)$, 
    as the set of all points of density one of the set $\{y\in \rn:\ \Deg(f,\opartial B,y)\neq 0\}$. 
\end{definition}

\begin{definition}[{$\INV$ condition}]\label{inv} 
    Let $f\in W^{1,n-1}(\Omega,\rn)\cap L^{\infty}(\Omega,\rn)$. We say that $f$ satisfies $\INV$ for a ball 
    $B\subset\subset\Omega$ 
    if 
    \begin{enumerate}[(i)]
      \item its trace on $\partial B$ is in $W^{1,n-1}(\partial B, \rn)\cap L^{\infty}(\partial B, \rn)$;
    	\item $f(x)\in \topi (f,B)$ for a.e.\ $x\in B$;
    	\item $f(x)\notin \topi (f,B)$ for a.e.\ $x\in\Omega\setminus B$.
    \end{enumerate}
    We say that $f$ satisfies $\INV$ if for every $a\in\Omega$ there is $r_a>0$ such that for $\mathcal{H}^1$-a.e.~$r\in (0,r_a)$ it satisfies $\INV$ for $B(a,r )$. 
\end{definition}

\begin{remark}\label{r:inv}
    If $f$, in addition, satisfies $J_f>0$ a.e., then preimages of sets of measure zero
    are of measure zero and thus we can characterize the $\INV$ condition in a simpler way.
    Namely, such a mapping satisfies the $\INV$ condition for the ball $B\subset\subset\Omega$
    if and only if
    \begin{enumerate}[(i)] 
      \item its trace on $\partial B$ is in $W^{1,n-1}(\partial B, \rn)\cap L^{\infty}(\partial B, \rn)$;
    	\item $\Deg(f,B,f(x))\neq 0$ for a.e.\ $x\in B$;
    	\item $\Deg(f,B,f(x))=0$ for a.e.\ $x\in\Omega\setminus B$.
    \end{enumerate}
\end{remark}

\begin{definition}
Let $f\in W^{1,n-1}(\Omega,\rn)\cap L^{\infty}(\Omega,\rn)$. 
The distributional Jacobian \index{distributional Jacobian} of $f$ is the distribution defined by setting
$$
\Det Df( \varphi) := -\int_\Omega f_1(x) J (\varphi, f_2, ... , f_n)(x)\dx \qquad \text{ for all } \varphi \in C^{\infty}_C(\Omega),
$$
where $J ( \varphi, f_2, ... , f_n)$ is the classical Jacobian defined as the determinant of the Jacobi 
matrix $Dg$ of $g=(\varphi, f_2, \ldots, f_n)$.  
\end{definition}

We need the following lemmata from \cite{CDL}. They are stated there only for $n=3$ but it is clear from the proofs that everything works also in higher dimensions. Note that the definition used in \cite{CDL} is different from the one above, but since $f\in W^{1,n-1}(\Omega,\rn)\cap L^{\infty}(\Omega,\rn)$, we can show by the standard approximation argument that for our class of functions they coincide.

\begin{lemma}[Lemma 3.8, \cite{CDL}]\label{lem:img}
    Let $f \in W^{1,n-1}(\Omega,\rn) \cap L^{\infty}(\Omega,\rn)$, with $J_f \neq 0$ on $\Omega_{d}$, and choose $B \subset \Omega$ such that $f$ satisfies $\INV$ for $B$. Then $f(B) \subset \topi(f, B)$, and $f(\rn \backslash B) \subset \rn \setminus \topi(f, B)$.
\end{lemma}

\begin{lemma}[Lemma 4.3, \cite{CDL}]\label{lem:distributional_jacobian}
    Let $f \in W^{1,n-1}(\Omega,\rn) \cap L^{\infty}(\Omega,\rn)$. Suppose that condition $\INV$ holds for $f$, and that $J_f>0$ a.e.
    Then,
    \begin{enumerate}
        \item[{\rm (i)}]  $\Det Df \geq 0$, hence it is a Radon measure;
        \item[{\rm (ii)}] the absolutely continuous part of $\Det Df$ with respect to $\mathcal{L}^{n}$ has density $J_f$;
        \item[{\rm (iii)}] for every $a \in \Omega$ and for a.e.~$r \in (0,r_a)$,
        \begin{equation}\label{image_size_CDL}
            \Det Df (B(a, r))=|\topi(f, B(a, r))|.
        \end{equation}
    \end{enumerate}
\end{lemma}


\subsection{Minimizers of the tangential Dirichlet integral}\label{s:dirichlet}
In our main proof, we have a sphere $\partial B$ in $\rn$ and on this sphere we have a small $(n-2)$-dimensional circle which is a boundary of an open spherical cap $S\subset \partial B$. Our map $f$ is in $W^{1,n-1}$, therefore we can choose the sets so that $f$ is continuous on the $(n-2)$-dimensional circle $\overline S\setminus S$. Our mapping $f$ can have a big oscillation on $S$ so we need to replace it with a reasonable mapping. We do this by choosing a minimizer of the tangential Dirichlet energy over this cap $S$ which has the same value on the circle $\overline S\setminus S$. 
In fact, we need this even for more general shapes than spheres and circles.

We say that a relatively open set $S\subset \partial B$ satisfies the \textit{exterior ball condition}
if for each $z\in \overline S\setminus S$ there exists a ball $B(z',r)$ with
$z'\in\partial B$ such that $z\in \partial B(z',r)$ and $B(z',r)\cap S=\varnothing$. The following Theorem was shown in \cite[Theorem 2.10]{DHM}: 

\begin{thm}\label{minimizers} 
    Let $B\subset\rn$ be a ball. Let $S\subset\partial B$ be a connected relatively open subset of $\partial B$. 
    Let $T$ be the relative boundary 
    of $S$ with respect to $\partial B$. Suppose that $\diam S<\frac{r}{4n}$ and that $S$ satisfies
    the exterior ball condition. Let 
    $f=(f^1,\dots,f^n)\in W^{1,n-1}(\partial B,\rn)$ be continuous on $T$. Then there exists 
    a unique function $h=(h^1,\dots,h^n)\in C(\overline S,\rn)\cap W^{1,n-1}(S,\rn)$ such that each coordinate
    $h^i$ minimizes $\int_{S}|D_{\tau}u|^{n-1}\,d\haus^{n-1}$ among all functions $u\in f^i+W_0^{1,n-1}(S,\rn)$.
    We have $h=f$ on $T$, 
    the function $h$ satisfies the estimate
    \begin{equation}\label{osc}
        \diam h(\overline S)\leq \sqrt n\diam f(T) 
    \end{equation}
    and we have $\mathcal{L}^n(h(S))=0$. 
    Moreover, let $f_m$ be continuous and converge to $f$ uniformly on $T$, then $h_m$ converge to $h$ uniformly on $S$, where $h_m$ are minimizers corresponding to boundary values $f_m$. 
\end{thm}
\begin{proof}
Everything except $\mathcal{L}^n(h(S))=0$ was shown already in \cite[Theorem 2.10]{DHM}. 

It is standard that the change of variable formula holds for Sobolev mappings up to a null set (see \eqref{area3} and the paragraph after) and hence using $h\in W^{1,n-1}$ there is $N\subset S$ with $\haus^{n-1}(N)=0$ such that
$$
\haus^{n-1}\bigl(h(S\setminus N)\bigr)<\infty\text{ and hence }\mathcal{L}^n\bigl(h(S\setminus N)\bigr)=0. 
$$

We claim that $h$ is pseudomonotone, i.e.\ there is $C>0$ such that for each spherical cap $A\subset S$ we have 
$$
\diam h(\overline A)\leq C\diam f(\partial A)
$$ 
(here $\partial A$ is the relative $(n-2)$-dimensional boundary with respect to $\partial B$). This fact follows from \eqref{osc}, i.e.,\ we consider the corresponding minimizer on $A$ with respect to boundary data $h|_{\partial A}$. By the uniqueness of this minimizer we obtain that $h|_{A}$ is this minimizer and \eqref{osc} holds for $\overline{A}$ and $\partial A$ (instead of $\overline{S}$ and $T$) gives us the pseudomonotonicity. 
Let us now consider a mapping $g:=P\circ h\colon S\to\er^{n-1}$, where $P$ is the projection to the hyperplane $\{x_1=0\}$. 
It is easy to see that $g\in W^{1,n-1}$ and that $g$ is continuous and pseudomonotone. By the result of Mal\'y and Martio \cite[Theorem A]{MM} and $\haus^{n-1}(N)=0$ we obtain that 
$$
\haus^{n-1}\bigl(g(N)\bigr)=0\text{ and hence }\mathcal{L}^n\bigl(h(N)\bigr)=0. 
$$
\end{proof}


\section{Proof of Theorem~\ref{main}: $\INV$ condition}\label{three}

\begin{proof}[Proof of Theorem \ref{main}: $\INV$ condition]
Assume on the contrary that $f$ does not satisfy the $\INV$ condition. 
Then we can find a center $c$ such that for $\haus^1$-positively many radii $r>0$ our $f$ maps either something from $B(c,r)$ outside of $\topi(f,B(c,r))$ or something outside of 
$B(c,r)$ inside of $\topi(f,B(c,r))$. We treat the first case in detail and at the end we briefly explain the analogous second case.  

{\underline{Step 1. Outline of the proof}:} 
We assume that $f$ does not satisfy $\INV$ and hence there is $c\in\Omega$ such that the set 
\begin{equation}\label{covering}
    \left\{r:\ B(c,r)\subset\Omega,\ \exists V_r\subset B(c,r)\text{ with }|V_r|>0\text{ and }\Deg(f,B(c,r),f(x))=0\text{ for all }x\in V_r\right\}
\end{equation}
has positive (one-dimensional) measure. 

Let us now briefly explain the idea of the proof. We know that $f_m$ converge to $f$ weakly in $W^{1,n-1}$ and thus up to a subsequence strongly in $L^{n-1}$ and a.e.
We can thus imagine that $f_m$ is really close to $f$ both on some fixed $\partial B(c,r)$ and on $V_r$. 
The situation is illustrated in Fig.~\ref{fig_bubble}.

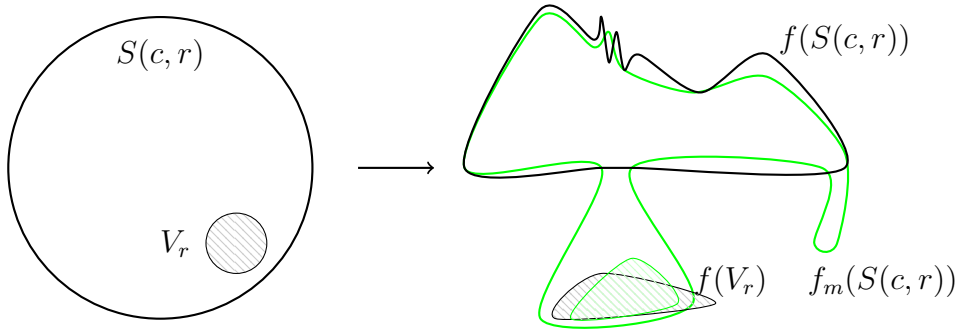
\begin{figure} [h t p]
\phantom{a}
\begin{tikzpicture}
    \draw[thick] (-4,0) circle(2);
    \node at (-4,1.5) {$S(c,r)$};
    \draw [black,pattern=north west lines, pattern color= {rgb:black,1;white,4}]  (-3,-1) circle(0.4);
    \node at (-3.8,-1) {$V_r$};

    \draw[->, thick] (-1.4,0)--(-0.4,0);

    \draw [thick, green] plot [smooth cycle] coordinates {(0,0) (1,2) (1.7,1.6) (1.9,1.8) (2.1,1.3) (3,1) (4,1.2) (5,0.2) (4.9,-1) (4.6, -1) (4.7,0) (2.2,0) (3,-1.8) (1,-2) (1.8,0)};
    \draw [thick, black] plot [smooth cycle] coordinates {(0,0) (1,2.1) (1.7,1.7) (1.8,2) (1.9, 1.4) (2, 1.8) (2.1,1.3) (2.3,1.5) (3.1,1) (4,1.5) (5,0) (2.2,0)  (1.8,0)};
    \draw [black, pattern=north west lines,  pattern color= {rgb:black,1;white,4}] plot [smooth cycle] coordinates { (3.3,-1.8) (1.8,-1.4) (1.2,-2)};
    \draw [green, pattern=north west lines,  pattern color= {rgb:green,1;white,4}] plot [smooth cycle] coordinates { (2.8,-1.8) (2.2,-1.2) (1.4,-2)};

    \node at (5.5,-1.5) {$f_m(S(c,r))$};
    \node at (3.5,-1.5) {$f(V_r)$};
    \node at (5,1.7) {$f(S(c,r))$};
\end{tikzpicture}
\caption{Behaviour of mappings $f$ (in black) and $f_m$ (in green) on $S(c,r)$ and $V_r$.}\label{fig_bubble}
\end{figure}

That means $f(V_r)$ is outside of 
$\topi(f,B(c,r))$ but $f_m(V_r)$ lies inside $f_m(\partial B(c,r))$ since $f_m$ is a homeomorphism. It follows that $f_m(\partial B(c,r))$ is close to 
$f(\partial B(c,r))$ in most of the places but it makes a ``bubble'' (or several bubbles) around $f_m(V_r)$ which is really close to $f(V_r)$. 

Firstly, we define those bubbles and then we show that the number of bubbles that contain a big part of $f_m(V_r)$ is uniformly bounded and hence in one of them we have a big portion of the volume of $f_m(V_r)$ (and that $|f_m(V_r)|\geq C$ using \eqref{varphi} and \eqref{key}). Since $f_m(V_r)$ is quite big (in one of the big bubbles) we obtain that the boundary of the bubble has big $\haus^{n-1}$ measure. However, this boundary is essentially image of some very small set (as small as we wish for $m$ big enough) on $\partial B$ under $f_m$ and using \eqref{estimate} we obtain that the integral of $|\cof Df_m|$ over this small set is big. 
This contradicts the equiintegrability of $|\cof Df_m|$ which results in $\sup_m \int_{\partial B(c,r)}A(|\cof Df_m|)\,d\haus^{n-1}<\infty$.

{\underline{Step 2. Replacement of $f$ on $\partial B(c,r)$ with continuous $g$ that has similar degree}:} 

We need to apply plenty of techniques and results developed in \cite{DHM}. For the convenience of the reader, we include most of the details in the current proof. 

We fix $B(c,R)\subset\subset\Omega$. 
Since $f_m$ converge weakly in $W^{1,n-1}$, which is compactly embedded into $L^{(n-1)^*-\epsilon}$, and so into $L^{n-1}$,
we obtain that $f_m$ converge to $f$ in $L^{n-1}(B(c,R))$.
Up to a subsequence we can thus assume that $f_m\to f$ pointwise a.e.\ and by Lemma \ref{jfnonzero} we obtain $J_f\neq 0$ a.e.
We fix $r\in(0,R)$ (and pass to a subsequence if necessary, see e.g.~\cite[Lemma 2.9]{MS}) such that 
$$
    f_m\to f\text{ weakly in }W^{1,n-1}(\partial B(c,r),\rn)\text{ and }\haus^{n-1}\text{-a.e.~on }\partial B(c,r)
$$  
and such that there exists a constant $C_2$ so that
\begin{equation}\label{prvniodhad}
    \int_{\partial B(c,r)}(|D_{\tau} f|^{n-1}+|D_{\tau} f_m|^{n-1})\,d\haus^{n-1} < C_2\text{ for all } m\in\en.
\end{equation}
Moreover, using Theorem \ref{chm} we can assume that all $f_m\in W^{1,n-1}(\partial B(c,r),\rn)$ satisfy the $\N$ condition on $\partial B(c,r)$. Analogously to the proof of \cite[Lemma 2.9]{MS} we use the Fatou Lemma
and \eqref{key} to deduce
$$
\int_0^R\liminf_{m\to\infty}\int_{\partial B(c,r)}A(|\cof Df_m|)\,d\haus^{n-1} \,d\varrho \leq \liminf_{m\to\infty}\int_0^R\int_{\partial B(c,r)}A(|\cof Df_m|)\,d\haus^{n-1} \,d\varrho <\infty.
$$
Choosing further $r$ so that the $\liminf$ on the lefthand side is finite and thus (passing again to a subsequence) we have
\begin{equation}\label{druhyodhad}
    \int_{\partial B(c,r)}A\bigl(|\cof Df_m|\bigr)\,d\haus^{n-1} < C_2\text{ for all } m\in\en.
\end{equation}
We set $B:=B(c,r)$ and choose $\ep>0$ small enough with the exact value to be specified later. 
Find 
$\rho\in (0,\,\min\{\frac{1}{16n}r,\frac{\epsilon}{2}\})$ 
such that for each $z\in\partial B$ we have
\begin{equation}\label{ahoj}
    \int_{\partial B\cap B(z,2\rho)}|D_{\tau} f|^{n-1}\,d\haus^{n-1}
    < \ep^{n-1}
\end{equation}
(we can do that since the integral over the whole $\partial B$ is finite). 
For each fixed $z\in \partial B$ we find $\rho_z\in (\rho,2\rho)$ such that 
\begin{equation}\label{howrho}
    \rho\int_{\partial B\cap \partial B(z,\rho_z)}|D_{\tau} f|^{n-1}\,d\hauso
    <\ep^{n-1},
\end{equation} 
which is possible because the length of $(\rho,2\rho)$ is $\rho$, combined with \eqref{ahoj}. 
Moreover, we can also choose $\rho_z$ such that $f_m\to f$ occurs $\hauso$-a.e.~on $\partial B\cap \partial B(z,\rho_z)$ and that 
$$
    \liminf_{m\to\infty}\|f_m\|_{W^{1,n-1}(\partial B\cap \partial B(z,\rho_z),\rn)}<\infty. 
$$
It follows that up to a subsequence (depending on $z$ and $\rho_z$, see e.g.~\cite[Lemma 2.9]{MS})
\begin{equation}\label{funiformly0}
    f_m\to f\text{ weakly in }W^{1,n-1}\text{ and also uniformly on }\partial B\cap \partial B(z,\rho_z). 
\end{equation}
Note that on the $(n-2)$-dimensional space $\partial B\cap \partial B(z,\rho_z)$ we have embedding into H\"older functions $W^{1,n-1}\hookrightarrow C^{0,1-\frac{n-2}{n-1}}$, thus $f$ is continuous there 
and we have the estimate 
\begin{equation}\label{morrey}
    \diam f(\partial B\cap \partial B(z,\rho_z))
    \leq C (\rho_z)^{1-\frac{n-2}{n-1}}\left(\int_{\partial B\cap \partial B(z,\rho_z)}|D_{\tau} f|^{n-1}\,d\hauso\right)^{\frac{1}{n-1}}
    \le C_3 \ep.
\end{equation}
Using the Vitali type covering, we find $B_j=B(z_j,\rho_j)$ such that $\rho_j=\rho_{z_j}$,
$$
    \partial B\subset \bigcup_jB(z_j,\rho_j)
$$
and the balls $B(z_j,\frac15\rho_j)$ are pairwise disjoint. Here $j=1,\dots,j_{\max}$.
Furthermore, the balls in the Vitali covering theorem are chosen inductively so we can also assume using \eqref{funiformly0} 
that for a subsequence (chosen in a diagonal argument) 
\begin{equation}\label{funiformly}
    f_m\to f\text{ weakly in }W^{1,n-1}\text{ and uniformly on } 
    \partial B\cap \partial B(z_j,\rho_j)\text{ for each }j.
\end{equation} 
Given $j$, denote 
$$
    S_j=\partial B\cap B_j\setminus \bigcup_{l<j}\overline{B_l}.
$$
Note that $S_j$ satisfies the exterior ball condition of Subsection \ref{s:dirichlet}. 
Let $T_j$ denote the relative boundary of $S_j$ with respect to $\partial B$.

For each $j$ we define $h_j$ on $S_j$ such that $h_j$ minimizes 
coordinate-wise 
the 
tangential $(n-1)$-Dirichlet integral among functions with boundary data $f$ on $T_j$ (see Theorem \ref{minimizers}). We define $h_j=f$ on $\partial B\setminus S_j$. 
Also we define the function $g$ on $\partial B$ as $g=h_j$ on 
each $\overline {S_j}$.
Set (see Fig. \ref{fig_F_j})
\begin{equation*}
\begin{aligned}
    F&=\{y\in \Omega'\colon \Deg(f,B,y)\ne \deg(g,B,y) \},\\
    F_j&=\{y\in  \Omega'\colon \Deg(f,B,y)\ne \Deg(h_j,B,y) \}.
\end{aligned}
\end{equation*}
Let us recall that by Theorem \ref{minimizers} we have $\mathcal{L}^n(g(S_j))=0$ and hence Remark \ref{Deg=deg} gives us $\deg g = \Deg g$.

\begin{figure}[h t p]
\begin{tikzpicture}[scale=1.5]
    \draw [thick, black,  fill= gray] plot [smooth cycle] coordinates {(0,0) (1,2.1) (1.7,1.7) (1.8,2) (1.9, 1.4) (2, 1.8) (2.1,1.3) (2.3,1.5) (3.1,1) (4,1.5) (5,0) (2.2,0)  (1.8,0)};
    \draw [white, fill=white] {(-0.2, -0.2)--(-0.2,2.2)--(4.6, 2.2)--	(4.6,0.9)-- (5,0)--(5,-0.2)--(-0.2, -0.2)};
    \draw [thick, black] plot [smooth cycle] coordinates {(0,0) (1,2.1) (1.7,1.7) (1.8,2) (1.9, 1.4) (2, 1.8) (2.1,1.3) (2.3,1.5) (3.1,1) (4,1.5) (5,0) (2.2,0)  (1.8,0)};
    \draw [dashed, thick, red] {(0,0)node {$\times$}--(0.4, 1.1) node {$\times$} --(1,2.1)node {$\times$}-- (1.7,1.7)node {$\times$}--  (2.3,1.5)node {$\times$}-- (3.1,1)node {$\times$} --(4,1.5)node {$\times$}--(4.6,0.9)node {$\times$}-- (5,0)node {$\times$} --(4,-0.1)node {$\times$}--(2.5,0)node{$\times$}--  (1.8,0)node {$\times$}--(0,0)};

    \node at (0.6,0.3) {$g(\partial B)$};
    \node at (2.3,2.1) {$f(\partial B)$};
    \node at (5.2,0.3) {$F_j$};
\end{tikzpicture}
\caption{Behaviour of mappings $f$ (in black) and $g$ (red) on $\partial B$ in 2D representation. $T_j$ corresponds to points on $f(\partial B)$ (of course in $\rn$ they are $(n-2)$-dimensional),
$g$ is represented by dashed lines connecting these points (of course these are minimizers of $(n-1)$-energy in higher dimensions and not lines) and the gray set $F_j$ is created ``between'' $g(S_j)$ and $f(S_j)$.}\label{fig_F_j}
\end{figure}
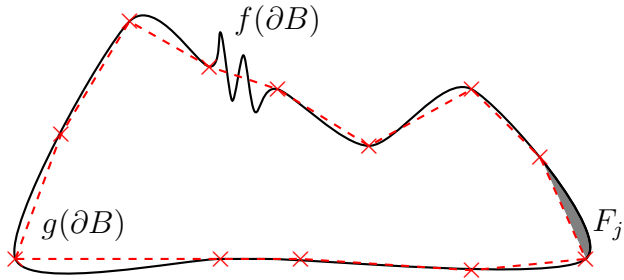

It is not difficult to find out that 
$$
    y\in \bigcup_jF_j \quad \text{for a.e.\ }y\in F
$$
(this can be viewed e.g.~by using \eqref{weakdegree}). 
Now, by \eqref{odhad}, \eqref{ahoj}, and the minimizing property
$\int_{S_j}|D_{\tau}h_j|^{n-1}\,d\haus^{n-1}\leq C\int_{S_j}|D_{\tau}f|^{n-1}\,d\haus^{n-1}$ we have
\begin{equation}\label{degg}
\begin{aligned}
    \sum_{j}|F_j|&\le C
    \sum_{j}\Big(\int_{S_j}(|D_{\tau} f|^{n-1}+|D_{\tau} h_j|^{n-1})\,d\haus^{n-1}\Big)^{\frac{n}{n-1}}
    \\&
    \le C\sum_{j}\Big(\int_{S_j}|D_{\tau} f|^{n-1}\,d\haus^{n-1}\Big)^{\frac{n}{n-1}}\\
    &\le C\ep
    \sum_{j}\int_{S_j}|D_{\tau} f|^{n-1}\,d\haus^{n-1}
    \le CC_2\ep.
\end{aligned}
\end{equation}
It follows that $f$ and $g$ have the same degree up to a very small set. It is more convenient for us to work with $g$ since for this continuous mapping on $\partial B$ we can use the classical degree $\deg$ and not $\Deg$ as for $f$. 

{\underline{Step 3. Replacement of $f_m$ on $\partial B(c,r)$ with continuous $g_m$ that is close to $g$}:} 

From Theorem \ref{minimizers} we  know that $\mathcal{L}^n(h_j(S_j))=0$ for each $j$ and thus $|g(\partial B)|=0$. 
It follows that we can find a compact set $H\subset\Omega'\setminus g(\partial B)$ such that  
\begin{equation}\label{defK}
    \Omega'\setminus H<\Phi(\tfrac{1}{10}|V_r|),
\end{equation}
where $\Phi$ comes from Lemma \ref{l:reverse}. 
For each $m\in\en$ and $j\in\{1,\dots,j_{\max}\}$ 
let $g_{m,j}$ be defined in $S_j$ as the 
coordinate-wise minimizer of the 
$(n-1)$-Dirichlet integral among functions with boundary data $f_m$ on $T_j$.
We define $g_{m,j}$ as
$f_m$ on $\partial B\setminus S_j$. We also define
$g_m$ on $\partial B$ as $g_{m,j}$ on each $\overline{S_j}$.

Since $f_m\to f=g$ uniformly on $T_j$ by \eqref{funiformly}, we have $g_m\to g$ uniformly on $\partial B$ using Theorem \ref{minimizers}.
Hence we find $m\in\en$ such that $g_m(\partial B)$ 
does not intersect 
$H$ and 
\begin{equation}\label{W'}
    \deg (g_m,B,\cdot)=\deg (g,B,\cdot)\quad\text{in }H.
\end{equation}

Also, we require
\begin{equation}\label{new}
    |f_m-f|=|g_m-g|< \ep\quad\text{on all }T_j.
\end{equation}
With the help of \eqref{osc} and \eqref{morrey} (which holds also for $T_j$) this implies that
$$
    |g_m-g|<C\epsilon\text{ on }\partial B. 
$$
Similarly as in Fig. \ref{fig_F_j} (but using $f_m$ instead of $f$), we define 
\eqn{defEj}
$$
\begin{aligned}
    E  &=\{y\in \Omega'\colon \deg(f_m,B,y)=1\ne \deg(g_m,B,y) \},\\
    E_j&=\{y\in  \Omega'\colon \deg(f_m,B,y)=1\ne \deg(g_{m,j},B,y) \},
\end{aligned}
$$
see Fig. \ref{fig_E_j}.

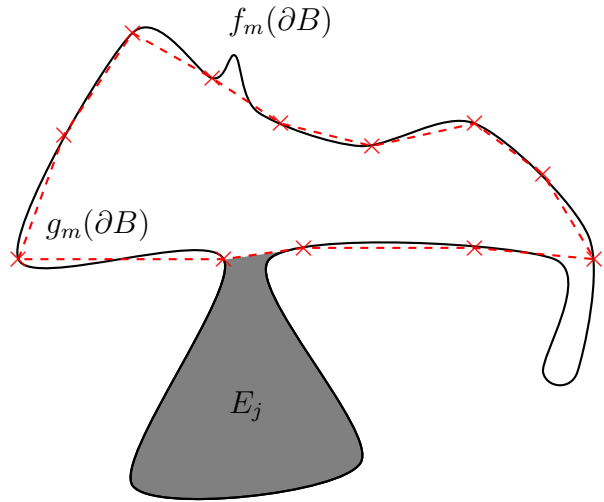
\begin{figure}[h t p]
\begin{tikzpicture}[scale=1.5]
    \draw [thick, black,  fill= gray] plot [smooth cycle] coordinates {(0,0) (1,2) (1.7,1.6) (1.9,1.8) (2.1,1.3) (3,1) (4,1.2) (5,0.2) (4.9,-1) (4.6, -1) (4.7,0) (2.2,0) (3,-1.8) (1,-2) (1.8,0)};
    \draw [white, fill=white] {(-0.2, -0.2)--(-0.2,2.2)--(5.6, 2.2)--	 (5.6,-2)--(4,-2)--(2.5,0.1)--(1.8,0)--(-0.2, -0.2)};
    \draw [thick, black] plot [smooth cycle] coordinates {(0,0) (1,2) (1.7,1.6) (1.9,1.8) (2.1,1.3) (3,1) (4,1.2) (5,0.2) (4.9,-1) (4.6, -1) (4.7,0) (2.2,0) (3,-1.8) (1,-2) (1.8,0)};
    \draw [dashed, thick, red] {(0,0)node {$\times$}--(0.4, 1.1) node {$\times$} --(1,2)node {$\times$}-- (1.7,1.6)node {$\times$}--  (2.3,1.2)node {$\times$}-- (3.1,1)node {$\times$} --(4,1.2)node {$\times$}--(4.6,0.75)node {$\times$}-- (5.05,0)node {$\times$} --(4,0.1)node {$\times$}--(2.5,0.1)node{$\times$}--  (1.8,0)node {$\times$}--(0,0)};

    \node at (0.7,0.3) {$g_m(\partial B)$};
    \node at (2.3,2.1) {$f_m(\partial B)$};
    \node at (2,-1.3) {$E_j$};
\end{tikzpicture}
\caption{Behaviour of mappings $f_m$ (in black) and $g_m$ (red) on $\partial B$ in 2D representation.}\label{fig_E_j}
\end{figure}

Let us note that these sets $E_j$ are exactly those bubbles discussed in the first step (see Fig. \ref{fig_bubble}). 
Then
$$
    y\in \bigcup_jE_j \quad \text{for a.e.\ }y\in E.
$$
Using \eqref{odhad} and the minimizing property
$\int_{S_j}|D_{\tau} g_{m,j}|^{n-1}\,d\haus^{n-1}\leq \int_{S_j}|D_{\tau} f_m|^{n-1}\,d\haus^{n-1}$, 
we obtain
\begin{equation}\label{sizeej}
    |E_j|^{1-\frac1n}\le C\int_{S_j}|D_{\tau} f_m|^{n-1}\,
    d\haus^{n-1}.
\end{equation}

{\underline{Step 4. Not that many big bubbles where $f_m$ and $g_m$ have different degree}:} 

Choose $a>0$ and set 
\begin{equation*}
\begin{aligned}
    J^+=\{j\colon \int_{S_j}|D_{\tau} f_m|^{n-1}\,d\haus^{n-1}>a\}, \\
    J^-=\{j\colon \int_{S_j}|D_{\tau} f_m|^{n-1}\,d\haus^{n-1}\le a\}. 
\end{aligned}
\end{equation*}
Note that \eqref{sizeej} implies that $|E_j|$ are small for $j\in J^{-}$. 
Hence using \eqref{prvniodhad}
\begin{equation}\label{rr}
\begin{aligned}
    \sum_{j\in J^-}|E_j|&\le C\sum_{j\in J^-}\Big(\int_{S_j}|D_{\tau} f_m|^{n-1}\,d\haus^{n-1}\Big)^{\frac{n}{n-1}}
    \\&\le Ca^{\frac1{n-1}}\sum_{j\in J^-}\int_{S_j}|D_{\tau} f_m|^{n-1}\,d\haus^{n-1}
    \\&
    \le Ca^{\frac1{n-1}}\int_{\partial B}|D_{\tau} f_m|^{n-1}\,d\haus^{n-1}\le C_4 a^{\frac1{n-1}},
\end{aligned}
\end{equation}
where $C_4=CC_2$. We fix $a$ such that 
\begin{equation}\label{rrr}
    C_4a^{\frac1{n-1}}\le \Phi\left(\tfrac{1}{10}|V_r|\right).
\end{equation}
We set 
$$
    W=f_m^{-1}\Big(\bigcup_{j\in J^-}E_j\Big).
$$
and using Lemma \ref{l:reverse}, 
\eqref{rr} and \eqref{rrr} we obtain 
$$
\Phi(|W|)\leq |f_m(W)|= \Bigl|\bigcup_{j\in J^-} E_j\Bigr|\leq C_4 a^{\frac{1}{n-1}}\leq \Phi\left(\tfrac{1}{10} |V_r|\right).
$$
From the monotonicity of $\Phi$ we get
\begin{equation}\label{Ej}
    |W|\leq\tfrac{1}{10}|V_r|.
\end{equation}

From \eqref{prvniodhad} we have
\begin{equation}\label{defM}
    \# J^+\le M:=\frac{C_2}{a}.
\end{equation}
It follows that we have only boundedly many $E_j$, $j\in J^+$, where the size of the bubble $|E_j|$ could be big. This bound depends on $|V_r|$ and $\Phi$ (and hence $\varphi$) and $C_2$ 
 but it does not depend on $m$ nor on $\epsilon$. 
We could have plenty of other small bubbles $E_j$, $j\in J^-$, but \eqref{Ej} implies that the union of their preimages is really small. 

{\underline{Step 5. Big part of $f_m(V_r)$ lies in big bubbles}:} 

Set 
$$
    Y=\{x\in V_r\setminus W\colon \deg(g_m,B,f_m(x))= 0\}.
$$
With the help of \eqref{W'} we have 
\begin{equation*}
\begin{aligned}
    V_r\setminus Y
    \subset &W\cup \{x\in V_r\colon\Deg(f,B,f(x))\neq 0\}\cup \\
    &\cup\{x\in V_r\colon \deg(g_m,B,f_m(x))\neq 0,\ \Deg(f,B,f(x))= 0\}\\
    \subset &W\cup \{x\in V_r\colon\Deg(f,B,f(x))\neq 0\}\cup\\ 
    &\cup\{x\in V_r\colon\deg(g,B,f(x))\neq \Deg(f,B,f(x))\} \cup\\
    &\cup\{x\in V_r\colon\deg(g,B,f_m(x))\neq \deg(g,B,f(x))\}\cup \{x\in V_r\colon f_m(x)\notin H\}.\\
\end{aligned}
\end{equation*}
From \eqref{Ej} and \eqref{covering} we obtain
$$
    \bigl|W\cup \{x\in V_r\colon\Deg(f,B,f(x))\neq 0\}\bigr|\leq \frac{1}{10}|V_r|
$$
as the second set is empty. 
Using \eqref{degg} we have 
$$
    \bigl|\bigl\{y\in\Omega'\colon\deg(g,B,y)\neq \Deg(f,B,y) \bigr\}\bigr|=|F|\leq C C_2\epsilon.
$$

Using Lemmata \ref{Ninv} and \ref{jfnonzero} we obtain that (for $\epsilon$ small enough)
$$
\bigl|\bigl\{x\in V_r\colon\deg(g,B,f(x))\neq \Deg(f,B,f(x))\bigr\}\bigr|\leq \frac{1}{10}|V_r|.
$$

Since the sets $\{y\colon \deg(g,B,y)=0\}$ and $\{y\colon \deg(g,B,y)=1\}$ are open and 
$f_m\to f$ a.e., we can take 
$m$ so large that 
\begin{equation}\label{by a.e.}
    \bigl|\bigl\{x\in V_r\colon\deg(g,B,f_m(x))\ne \deg(g,B,f(x))\bigr\}\bigr|<\frac{1}{10}|V_r|. 
\end{equation}
Finally using \eqref{defK} and \eqref{reverse} (as in \eqref{Ej}) we obtain 
$$
    |\{x\in V_r\colon f_m(x)\notin H\}|\leq \frac{1}{10}|V_r|
$$
and all these inequalities together give us    
\begin{equation}\label{cil}
    |V_r\setminus Y|\le \frac{1}{2}|V_r|.
\end{equation}

It follows that for many points $x\in V_r$ we have
$$
    \deg(g_m,B,f_m(x))=0,
$$
but 
$$
    \deg(f_m,B,f_m(x))=1
$$
since $f_m$ is a homeomorphism and $x\in V_r\subset B$. Therefore
\begin{equation}\label{klic}
    \left|\{x\in V_r:\ f_m(x)\in E_j\text{ for some }j\in J^+\}\right|\geq \frac{1}{2}|V_r|. 
\end{equation}

{\underline{Step 6. Integral $\int_{S_j} |\cof Df_m|$ is big on a small set $S_j$ for some $j\in J^+$}:} 

Using \eqref{defM} and \eqref{klic} we fix $j\in J^+$ such that for
$$
U:=\{x\in V_r:\ f_m(x)\in E_j\}\text{ we have } \bigl|U\bigr|\geq \frac{1}{2\# J^{+}}|V_r|\geq C|V_r|.
$$
From Lemma \ref{l:reverse} \eqref{reverse} we obtain that
\eqn{velky}
$$
|f_m(U)|\geq 2\delta,
$$
where $\delta$ is constant which does not depend on $m$ or $\epsilon$. 

From the definition of $E_j$ (see \eqref{defEj} and Fig. \ref{fig_E_j}) we obtain that $E_j$ is an open set and 
$$
\partial E_j\subset f_m(S_j)\cup g_{m}(S_j).
$$

We know that (see \eqref{morrey}, \eqref{new} and \eqref{osc})
$$
    \diam(g_m(\overline{S_j}))\leq C\epsilon
$$
and thus we can find a ball $B_0$ of radius $C\epsilon$ such that $g_m(\overline{S_j})\subset B_0$. Now the set 
$$
\tilde{E}_j:=E_j\setminus \overline{B_0}
$$
is open,  
$$
\partial \tilde{E}_j\subset f_m(S_j)\cup \partial B_0. 
$$
and using \eqref{velky} we obtain that 
$$
|\tilde{E}_j|\geq 2\delta-C\epsilon^n>\delta
$$
once $\epsilon$ is small enough. 

It is not difficult to show that the set $\tilde{E}_j$ has finite perimeter, and therefore we can use
 the isoperimetric inequality \eqref{isoperimetric} 
and Theorem \ref{chm} to get
$$
\begin{aligned}
\delta^{1-\frac{1}{n}}&\leq |\tilde{E}_j|^{1-\frac{1}{n}}\leq C \haus^{n-1}(\partial \tilde{E}_j)\leq C\bigl(\haus^{n-1}(f_m(S_j))+\haus^{n-1}(\partial B_0)\bigr)\\
&\leq C(r)\Bigl(\int_{S_j}|\cof Df_m| \,d\haus^{n-1} +C_0 \epsilon^{n-1}\Bigr). 
\end{aligned}
$$
It follows that for $\epsilon$ sufficiently small we get 
$$
\frac{1}{2}\delta^{1-\frac{1}{n}}\leq  C(r)\int_{S_j}|\cof Df_m| \,d\haus^{n-1}
$$
and this estimate on $S_j$ (with $\diam S_j\leq \epsilon$) clearly contradicts the uniform integrability of $|\cof Df_m|$ given by \eqref{druhyodhad}.

{\underline{Step 7. Something from outside of $B(c,r)$ goes inside $\topi(f,B(c,r))$}:} 

This case works analogously. We can find a ball $B(c,r)$ and $V_r\subset \Omega\setminus B(c, r)$ such that its big part is mapped inside the topological images of $B(c,r)$. Therefore, $f_m$ creates bubbles inside. We can define $F_j$, $F$, $E_j$ and $E$ in a similar way and conclude. 
\end{proof}

Recall that the symmetric difference of two sets $S,T\subset\rn$ is defined as
$$
S\triangle T:=(S\setminus  T)\cup(T\setminus S). 
$$

\begin{lemma}\label{lem:fk-ft}
Let $n\geq 3$, $\Omega,\Omega'\subset\rn$ be bounded domains
and let $\ff$  
satisfy \eqref{varphi} and \eqref{varphi2}.  
Let $f_m\in W^{1,n-1}(\Omega,\Omega')$, 
$m =0,1,2\dots$, 
be a sequence of 
homeomorphisms of 
$\overline\Omega$ onto $\overline{\Omega'}$ 
with $J_{f_m}> 0$ a.e.\ 
such that 
$$
\sup_m \int_{\Omega}\Bigl(|D f_m|^{n-1}+\ff(J_{f_m})\Bigr)\dx<\infty. 
$$
Assume further that $f_m=f_0$ on $\partial\Omega$ for all $m\in\en$. 
Let $f$ be a weak limit of $f_m$ in $W^{1,n-1}(\Omega,\rn)$ and $B\subseteq \Omega$ be a ball such that $f$ satisfies $\INV$ for $B$,
    \begin{itemize}
\item $f_m \to f$ weakly in $W^{1,n-1}(\partial B, \mathbb{R}^n)$,
\item $f_m\to f$ $\haus^{n-1}$-a.e.\ on $\partial B$,
\item $\int_{\partial B} \left(|D_{\tau} f|^{n-1}+|D_{\tau} f_m|^{n-1} \right)\,d\haus^{n-1} < C.$
    \end{itemize}
    Then it holds that
    \begin{equation}\label{eq:fk-ft}
        |f_k(B)\triangle \topi(f,B)|\overset{k\to\infty}{\to} 0.
    \end{equation}
\end{lemma}

\begin{proof}
Let us assume by contradiction that we have a subsequence $f_{m_k}$ such that 
$$ 
|f_{m_k}(B)\triangle \topi(f,B)|>4\lambda > 0.
$$ 
We show that then $f$ does not satisfy $\INV$.

For simplicity, we pass to that subsequence and keep the notation $f_m$. We consequentially choose a further subsequence such that $f_m(x)\to f(x)$ a.e.\ on $B$ and, again passing to subsequences if neccessary, given $\epsilon>0$ we find $g$ as in the proof of Theorem \ref{main} such that 
\eqn{qqqq}
$$
|\topi(f,B)\triangle \{y\in\Omega': \deg(g,B,y)\neq 0\} |\leq |\{y\in\Omega': \deg(g,B,y)\neq \Deg(f,B,y)\}|\leq \varepsilon.
$$
Let us split into two cases: either $ |\topi(f,B)\setminus f_m(B)|>2\lambda$ or $|f_m(B)\setminus \topi(f,B)|>2\lambda$ for infinitely many $m$ and thus we can assume that this is true for all $m$.

In the first case we find
$$
U_m\subseteq \Omega\setminus B\text{ such that } f_m(U_m)\subseteq \topi(f,B)\text{ and } |f_m(U_m)|>2\lambda.
$$ 
We can assume that $\epsilon<\lambda$ and hence we can find $U_m'\subseteq U_m$ such that 
$$
f_m(U_m')\cap \bigl(\topi(f,B)\triangle \{y\in\Omega': \deg(g,B,y)\neq 0\}\bigr)=\varnothing 
$$
and $|f_m(U_m')|>\lambda$. We know from \eqref{reverse2} that $|U_m'|>\Psi^{-1}(\lambda)$. Therefore, 
$$
|U|>\Psi^{-1}(\lambda)/2,\text{ where }
U=\limsup U_m'=\bigcap_{k=1}^{\infty}\bigcup_{m=k}^{\infty}U_m'.
$$ 
Using Theorem \ref{minimizers} (as $|h(S)|=0$ there) and \eqref{qqqq} we have 
$$
\bigl|\partial \{y\in\Omega': \deg(g,B,y)\neq 0\})\bigr|=0 \text{ and } 
\bigl|\topi(f,B)\triangle \{y\in\Omega': \deg(g,B,y)\neq 0\}\bigr|<\varepsilon
$$
and according to Lemmata \ref{jfnonzero} and \ref{Ninv} their preimages under $f$ are arbitrarily small (depending on $\varepsilon$). Hence we can set $\varepsilon$ to be small enough such that $|U'|>\Psi^{-1}(\lambda)/2$, where
$$
U'=U\setminus\Bigl[f^{-1}\Bigl(\partial \{y\in\Omega': \deg(g,B,y)\neq 0\}\cup \bigl(\topi(f,B)\triangle \{y\in\Omega': \deg(g,B,y)\neq 0\}\bigr)\Bigr)\Bigr].
$$
For every $x\in U'$ we have a subsequence $f_{m_k}$ such that $x\in U_{m_k}'$ and thus
$$
f_{m_k}(x)\in \topi(f,B)\cap \{y\in\Omega': \deg(g,B,y)\neq 0\}.
$$
Since $f_m\to f$ pointwise a.e., we have for a.e.\ $x\in U'$ that 
$$
f(x)\in \overline{\{y\in\Omega': \deg(g,B,y)\neq 0\}}.
$$
However, from the definition of $U'$ we know that actually 
$$
f(x)\in \{y\in\Omega': \deg(g,B,y)\neq 0\},
$$
and since 
$$
x\notin f^{-1}\bigl(\topi(f,B)\triangle \{y\in\Omega': \deg(g,B,y)\neq 0\}\bigr),
$$
we have that $f(x)\in \topi(f,B)$ for every $x\in U'$. That contradicts $\INV$.

We deal with the second case analogously. We find
$$
U_m\subseteq B\text{ such that } f_m(U_m)\cap \topi(f,B)=\varnothing\text{ and } |f_m(U_m)|>2\lambda.
$$ 
We define the following sets in the same way and arrive to 
$$
f_{m_k}(x)\in \bigl(\Omega'\setminus\topi(f,B)\bigr)\cap 
\bigl(\Omega'\setminus \{y\in\Omega': \deg(g,B,y)\neq 0\}\bigr).
$$
Again, $x$ is not in the $f$-preimage of 
$$
\partial \{y\in\Omega': \deg(g,B,y)\neq 0\}\text{ nor of }
\topi(f,B)\triangle \{y\in\Omega': \deg(g,B,y)\neq 0\},
$$ 
and therefore $f(x)\notin \topi(f,B)$ for a set of a positive measure, which contradicts $\INV$.
\end{proof}

\section{Proof of Theorem~\ref{main}: $\N$ condition, lower semicontinuity and injectivity a.e.  }

\subsection{Lusin $\N$ condition}\label{sec:Lusin}

\begin{lemma}\label{lem:N}
    Let $f_{m} \in W^{1,n-1}(\Omega,\mathbb{R}^n)$ be a sequence of homeomorphisms with $J_{f_m} >0$ a.e. such that $f_{m}$ satisfies the Lusin $\N$ condition and the sequence of Jacobians $J_{f_m}$ is equiintegrable.
     Assume that $f\in W^{1,n-1}(\Omega,\rn) \cap L^{\infty}(\Omega, \rn)$ is a weak limit of $f_m$ in $W^{1,n-1}(\Omega,\mathbb{R}^n)$ such that for every $a\in\Omega$ there is $r_a>0$ such that for $\mathcal{H}^1$-a.e.~$r\in (0,r_a)$ it satisfies \eqref{eq:fk-ft} and $\INV$ for $B(a,r )$.
    Then,
    \begin{enumerate} 
        \item[(i)] the distributional Jacobian $\Det Df \geq 0$ is a Radon measure;
        \item[(ii)] $\Det Df$ is absolutely continuous w.r.t.~Lebesgue measure
        : for any set $E\subset \Omega$ with $|E|=0$, it holds that $\Det Df (E) = 0$;
        \item[(iii)] $f$ satisfies the Lusin $\N$ condition.
    \end{enumerate}
\end{lemma}

\begin{proof}

    The first item $(i)$  is stated in Lemma~\ref{lem:distributional_jacobian} $(i)$.
    
    Let $E\subset\Omega$ with $|E|=0$ be given. Fix $\delta >0$ and let $c(n)$ be a constant from Besicovitch covering theorem. Since $E$ is a set of measure zero, there exists 
    an open set $U \subset \Omega$ such that $E\subset A$ and
    $|U| < \frac{\delta}{c(n)}$. 
    Consider a covering of $E$ by balls $B(a,\tilde{r}_{a})$ for all $a\in E$ such that $B(a,\tilde{r}_{a})\subset U$, $f$ satisfies $\INV$ for $B(a,\tilde{r}_{a})$ and \eqref{eq:fk-ft} holds on this ball. 
    By the Bezicovitch Theorem (e.g.~\cite[Theorem A.2]{HK}), we can find at most countable collection of balls $B_k:=B(a_k,\tilde{r}_{k})$ such that 
	$$
	E\subset \bigcup_{k} B_k\subset U\text{ and }\bigcup_k B_{k} = \bigcup_{j=1}^{c(n)} \bigcup_{B_i \in A_j} B_{i},
	$$
    where subcollections $A_j$ consists of disjoined balls $B_{i}$ and a constant $c(n)$ depends only on the dimension $n$.

    Note that since \eqref{eq:fk-ft} is valid for one ball of the covering, it stays true for a finite union of such balls $\bigcup_{k=1}^{M} B_{k}$: 
    for any $\varepsilon>0$ and $m$ big enough it holds that
    \begin{equation}\label{eq:N1}
        \Bigl|f_m\Bigl(\bigcup_{k=1}^{M} B_{k}\Bigr) \triangle  \Bigl(\bigcup_{k=1}^{M} \topi\left(f,B_{k}\right)\Bigr)\Bigr| \leq \frac{\varepsilon}{3 c(n)}.
    \end{equation}
    Then, by Lemma~\ref{lem:distributional_jacobian} $(iii)$ we have
    \begin{equation*}
    \begin{aligned}
		\Det Df\Bigl(\bigcup_{k=1}^{M} B_{k}\Bigr) & =
        \Det Df\Bigl(\bigcup_{j=1}^{c(n)} \bigcup_{B_i \in A_j,i\leq M} B_{i}\Bigr) \leq 
        \sum_{j=1}^{c(n)}\Det Df\Bigl(\bigcup_{B_i \in A_j, i\leq M} B_{i}\Bigr) \\
        & \leq \sum_{j=1}^{c(n)}
        \Bigl|\bigcup_{B_i \in A_j, i\leq M}\topi (f, B_i)\Bigr|
        \leq c(n)
        \Bigl|\bigcup_{k=1}^{M}\topi (f, B_k)\Bigr|.
    \end{aligned}
	\end{equation*}
    To prove $(ii)$, we fix $\varepsilon>0$ and $\delta>0$ such that
    \begin{equation}\label{eq:N2}
        \Psi(t) < \frac{\varepsilon}{3 \, 
        c(n)
        } \quad\text{ for any } t<\delta,
    \end{equation}
    where $\Psi$ is given by Lemma~\ref{l:reverse} (note that in the proof of this part of lemma we have used only equiintegrability of $J_{f_m}$ and Lusin $\N$ condition for $f_m$).
    Since $\bigcup_k B_k\subset U$ and $|U|<\delta$ we have using \eqref{reverse2}
    $$
    \sum_{k\in\mathbb{N}} \left|B_{k}\right| < \delta,\text{ and therefore }
    \Bigl|f_m\Bigl(\bigcup_{k=1}^{M} B_{k}\Bigr)\Bigr| < \dfrac{\varepsilon}{3 \,
        c(n)
        }\text{ for any }m\in\en.
    $$
    Relying on \eqref{eq:N1}--\eqref{eq:N2}, for $M$ big enough there exists $m$ such that
    \begin{equation*}
    \begin{aligned}
        \Det Df(E) & \leq 
        \Det Df\Bigl(\bigcup_{k\in\en} B_{k}\Bigr) \leq 
        \Det Df\Bigl(\bigcup_{k=1}^{M} B_{k}\Bigr) + \frac{\varepsilon}{3} 
        \leq c(n)
        \Bigl|\bigcup_{k=1}^{M}\topi (f,B_k)\Bigr| +\frac{\varepsilon}{3} \\
        & \leq 
        c(n)
        \Bigl|f_m\Bigl(\bigcup_{k=1}^{M} B_{k}\Bigr)\Bigr| + 
        c(n) \Bigl|f_m\Bigl(\bigcup_{k=1}^{M} B_{k}\Bigr) \triangle  \Bigl(\bigcup_{k=1}^{M} \topi\left(f,B_{k}\right)\Bigr)\Bigr| +\frac{\varepsilon}{3}
        \leq \varepsilon.
    \end{aligned}
    \end{equation*}

 For $(iii)$ it is enough to notice that
    $|f(B)| \leq |\topi f(B)|$ by Lemma \ref{lem:img}. 
    Hence
    \begin{equation*}
    \begin{aligned}
        |f(E)| &\leq  
        \Bigl|\bigcup_{k=1}^{M}\topi (f,B_k)\Bigr| +\frac{\varepsilon}{3} \\
        & \leq \Bigl|f_m\Bigl(\bigcup_{k=1}^{M} B_{k}\Bigr)\Bigr| 
        + \Bigl|f_m\Bigl(\bigcup_{k=1}^{M} B_{k}\Bigr) \triangle  \Bigl(\bigcup_{k=1}^{M} \topi\left(f,B_{k}\right)\Bigr)\Bigr| +\frac{\varepsilon}{3}
        \leq \varepsilon.
    \end{aligned}
    \end{equation*}
    Since $\varepsilon>0$ has been chosen arbitrary, we conclude
    $\Det Df(E) = |f(E)| =0$.

\end{proof}

\subsection{Lower semicontinuity}\label{sec:lsc}

The main obstacle to obtain the lower semicontinuity of $\F$ is to ensure weak convergence of Jacobians. One usually has to assume higher regularity of $f \in W^{1,n}$ (see e.g.~\cite{CM,DMS,M}) but this is not available for us. 


\begin{lemma}\label{lem:jacobian_convergence}
    Let $f_{m} \in W^{1,n-1}(\Omega,\mathbb{R}^n)$ be a sequence of homeomorphisms with $J_{f_m} >0$ a.e., such that $f_{m}$ satisfies the Lusin $\N$ condition and the sequence of Jacobians $J_{f_m}$ is equiintegrable.
     Assume that $f\in W^{1,n-1}(\Omega,\rn) \cap L^{\infty}(\Omega, \rn)$ is a weak limit of $f_m$ in $W^{1,n-1}(\Omega,\mathbb{R}^n)$ such that for every $a\in\Omega$ there is $r_a>0$ such that for $\mathcal{H}^1$-a.e.~$r\in (0,r_a)$ it satisfies \eqref{eq:fk-ft} and $\INV$ for $B(a,r )$. 
    Then there exists a subsequence $\{J_{f_{k_m}}\}_{m\in\en}$  of $\{J_{f_k}\}_{k\in\en}$ such that
    $J_{f_{k_m}} \rightharpoonup J_{f}$ weakly in $L^{1}(\Omega)$.
\end{lemma}
\begin{proof}
    By the Dunford--Pettis Theorem (e.g.~\cite[Theorem B.103]{L}) there exist a (non-relabeled) subsequence $J_{f_m}$ and a function $j \in L^1(\Omega)$ s.t.~$J_{f_m} \to j$ weakly in $L^{1}(\Omega)$, 
    and hence for any measurable set~$E\subset \Omega$
    $$
        \int_E J_{f_m}(x) \dx\to \int_E j(x) \dx.
    $$
    Let $B$ be such that $f$ satisfies the $\INV$ condition and \eqref{eq:fk-ft} for $B$, 
    then 
    by the area formula \eqref{area1} and \eqref{eq:fk-ft} we have  
    \begin{equation*} 
        \int_B J_{f_m}(x) \dx = |f_m(B)|\to |\topi(f,B)|. 
    \end{equation*}
    
        On the other hand, by Lemma~\ref{lem:distributional_jacobian} $(ii)$--$(iii)$ and Lemma~\ref{lem:N} $(ii)$, we know that 
    $$\Det f(B) = \int_B J_f(x) \dx = |\topi f(B)|.$$ 		
    Therefore, for all such balls, it holds that
    $$
        \int_B j(x) \dx= |\topi (f,B)| =\int_B J_f(x) \dx,
    $$
    which in turn implies 
    $j(x)=J_f(x)$ for a.e.~$x\in\Omega$.
\end{proof}

\begin{lemma}\label{previous}
    Under conditions of Theorem~\ref{main}, the functional $\F$ is lower semicontinuous with respect to weak convergence, i.e.,~\eqref{optimistic} holds.
\end{lemma}
\begin{proof}
Weak convergence $f_k\to f$ in $W^{1,n-1}$ implies weak convergence (up to subsequence) in $L^1$ of
    all minors $\det_l(Df_m)$ of order $l\leq n-2$ (see e.g.~\cite[Lemma 5.10]{Ri}):
    $$
        \|\det_l(Df_m)\|_{L^1(\Omega)} \leq C \|Df_m\|_{L^{n-1}}^{l}
        \leq C M^{\frac{l}{n-1}}.
    $$ 
    Weak convergence of $\det_{n-1}(Df_k) = \cof Df_k$ in $L^1$ follows by the standard argument, provided uniform integrability \eqref{A} (see \cite[Theorem 6.2]{Ball} or \cite[Theorem 7.5-1]{Ciar1988}). 
    The equiintegrability of $J_{f_m}$ follows from \eqref{varphi2} by the de la Val\'{e}e Poussin Theorem (e.g.~\cite[Theorem B.104]{L}).
   Moreover, conditions of Lemma~\ref{lem:fk-ft} are fulfilled (see Step 2 of the proof in Section~\ref{three}).
    Therefore, Lemmata~\ref{lem:fk-ft} and \ref{lem:jacobian_convergence} ensure $J_{f_m}\to J_f$ weakly in $L^1$,
 which now allows us to use De Giorgi Theorem \cite[Theorem 3.23]{Dac2008} to conclude the lower semicontinuity \eqref{optimistic} of $\F$.
\end{proof}


\subsection{Injectivity almost everywhere}\label{sec:injectivity}
One of the main reasons to consider the $\INV$ condition is that it implies injectivity a.e.~\cite[Lemma 3.7]{CDL}, \cite[Lemma 3.4]{MS}, as discussed in Introduction.
In the setting of this paper, we can say even more.

\begin{lemma}\label{inverse}
    Let conditions of Theorem~\ref{main} be fulfilled and
    let $h$ be a weak limit of $f_m^{-1}$ in $W^{1,1}(\Omega',\er^n)$.
    Then $h(f(x))=x$ for a.e.~$x\in\Omega$ and under additional assumption $|\partial\Omega'|=0$ we have $f(h(y))=y$ for a.e.~$y\in \Omega'$.
\end{lemma}

\begin{proof}
	From Theorem \ref{equiintegrable} we obtain that there is a subsequence of $f_m^{-1}$ which converges weakly in $W^{1,1}$ to some $h$ and we work with this subsequence here. 

    Recall that we know that $J_f(x)> 0$ for a.e.~$x\in\Omega$. 
    Since $f$ satisfies $\N$ we can use \eqref{area2} (for $A=f^{-1}(E)$) to obtain that 
		\eqn{Ninverse}
		$$
		|f^{-1}(E)|=0\text{ for every }E\subset \rn\text{ with }|E|=0.
		$$
    Therefore, for a.e.~$x\in\Omega$ we know that $f(x)$ is a Lebesgue point of $h$.

    Moreover, we claim that for a.e.~$x\in\Omega$ and every $\eta>0$ 
    there exists a ball $B=B(a,r)$ such that
    $f$ satisfies $\INV$ and \eqref{eq:fk-ft} for $B$ and 
    \begin{equation}\label{nice_points}
        x\in B, \: r<\eta, \text{ and }
        f(x)\in \topi(f,B)\text{ is a point of density $1$ of } \topi(f,B).
    \end{equation}
    Indeed, choose a countable set of balls
    $$
    \begin{aligned}
        \mathcal{B}:=\bigl\{B(c,r_i):\ &c\in\mathbb{Q}^n\cap \Omega,\ r_i \in [2^{-i-1},2^{-i})\text{ for all }i\in\en \\
        &\text{ and }f \text{ satisfies }\INV\text{ and }\eqref{eq:fk-ft} \text{ for }B(c,r_i)\bigr\}. \\
    \end{aligned}    
    $$
		For every $B_j\in \mathcal{B}$ we know that a.e.~point of $\topi(f,B_j)$ is a point of density $1$ and with the help of \eqref{Ninverse} we can find a null set $\Sigma_j$ such that 
    $$
        f(x)\in \topi(f,B_j)\text{ for each }x\in B_j\setminus \Sigma_j\text{ and }f(x) 
        \text{ is a point of density of }\topi(f,B_j). 
    $$
    Then $\Sigma:=\bigcup_j \Sigma_j$ is a null set and for every $x\in \Omega \setminus \Sigma$ we have \eqref{nice_points} for some ball $B_j$. 

    Let us first prove that $h(f(x))=x$ a.e.
    We pick $x$ such that $f(x)$ is a Lebesgue point of $h$ and  $f(x)$ is a  point of density of $\topi(f,B)$ for some ball $B=B(a,r)$ satisfying \eqref{nice_points}.
    Now we can find a ball $\tilde{B}$ around $f(x)$ so that 
    $$
        \frac{1}{|\tilde{B}|}\int_{\tilde{B}}|h(z)-h(f(x))|\,dz<r/2 
        \qquad\text{ and}\qquad
        \left|\bigl\{y\in\tilde{B}:\ y\in \topi(f,B)\bigr\}\right|>0.9|\tilde{B}|. 
    $$
    Using convergence $f_m^{-1}$ to $h$ in $L^{1}_{\loc}(\Omega',\rn)$ and \eqref{eq:fk-ft}, 
    we fix $m$ big enough so that 
    $$
        \int_{\tilde{B}} |f_m^{-1}(z)-h(z)|\,dz<r|\tilde{B}|/2
        \qquad\text{and}\qquad
        |f_m(B)\triangle \topi(f,B)|<0.1 |\tilde B|.
    $$ 
    Combining the estimates, we obtain
    $$
        \frac{1}{|\tilde{B}|}\int_{\tilde{B}} |f_m^{-1}(z)-h(f(x))|\,dz<r
        \qquad\text{and}\qquad
        \left|\bigl\{y\in\tilde{B}:\ y\in f_m(B)\bigr\}\right|>0.8|\tilde{B}|. 
    $$
    We claim that this implies
		\eqn{aha}
    $$
        h(f(x))\in B(a,4r), 
    $$
	 since otherwise we get a contradiction from 
	$$
    \begin{aligned}
    	\frac{1}{|\tilde{B}|}\int_{\tilde{B}} |f_m^{-1}(z)-h(f(x))|\,dz& \geq 
    	\frac{1}{|\tilde{B}|}\int_{\{y\in\tilde{B}:\ y\in f_m(B)\}} \bigl|h(f(x))-a-(f_m^{-1}(z)-a)\bigr|\,dz \\
    	& \geq 0.8 (4 r-r)>r. 
    \end{aligned}
	$$
    Since $r>0$ is chosen arbitrary, we conclude from \eqref{aha} that $h(f(x))=x$ for a.e.~$x\in\Omega$.

It is not difficult to see that $f(\Omega)\subset\overline{\Omega'}$. 
From Lemma \ref{lem:jacobian_convergence} and change of variables \eqref{area2} we know that 
$$
|\Omega'|=\lim_{k\to\infty}|f_k(\Omega)|=\lim_{k\to\infty}\int_{\Omega}J_{f_k}(x)\dx=\int_{\Omega}J_f(x)\dx= \int_{\rn}N(f,\Omega,y)\dy. 
$$ 
From the a.e.-injectivity \cite[Lemma 3.7]{CDL} of $f$ together with the $\N$ condition for $f$ we now obtain 
$$
|\Omega'|=\int_{\rn}N(f,\Omega,y)\dy=\int_{f(\Omega)}1\dy=|f(\Omega)|.
$$
Since $f(\Omega)\subset\overline{\Omega'}$ and $|\partial \Omega'|=0$ we obtain that a.e.~point $y\in \Omega'$ lies in $f(\Omega)$ and $N(f,\Omega,y)=1$ there. 
  
    The other equality $f(h(y)) = y$ for a.e.~$y\in \Omega'$ now follows easily. 
		We know $h(f(x))=x$ holds for a.e.~$x\in \Omega$ and that $f$ satisfies the $\N$ condition. Hence for a.e.~$y\in f(\Omega)$ we can pick $x\in \Omega$ such that $f(x)=y$ and  $h(f(x))=x$. Now
    $$
		f(h(y))=f(h(f(x)))=f(x)=y.
		$$
		Note that in this proof we do not need a Sobolev regularity but only $f_m^{-1} \to h$ in $L^{1}_{\loc}$.
\end{proof}

\begin{proof}[Proof of Theorem~\ref{main}]
Theorem \ref{main} now follows from result in Section \ref{three}, Lemma \ref{lem:N}, Lemma \ref{inverse}  and Lemma~\ref{previous}.
\end{proof}

\subsection{Counterexamples to lower semicontinuity}\label{counter}

The following example shows that one has to ask the condition $\N$ for $f_{m}$ to conclude lower semicontinuity of a quasiconvex functional, even if $\ff$ and $A$ satisfy \eqref{varphi} and \eqref{A}.

\begin{lemma}[Counterexample for lsc]\label{lem:example}
Let $p<n$, then there exist  $\ff$ and $A$ that satisfy \eqref{varphi} and \eqref{A} and homeomorphisms $f_m$, $f\colon[0,1]^n \to [0,1]^n$ 
such that $J_{f_m}$, $J_f >0$ a.e., \eqref{key} is fulfilled,
$f_m=id$ on $\partial([0,1]^n)$ and $f_m$ converge to $f$ weakly in $W^{1,p}([0,1]^n,\rn)$. 
However, $f_m$ does not satisfy the Lusin $\N$ condition for all $m\in\en$ and
$$
\int_{(0,1)^n} J_{f}(x) \dx>
\liminf_{m\to\infty} 
\int_{(0,1)^n} J_{f_m}(x) \dx.
$$
\end{lemma}

\begin{proof}
   Take any $A(t) \leq C t^{\beta}$, where $C>0$ and $\beta >1$ are some constants, and take $\varphi$ which behaves like an identity around $1$ and satisfies \eqref{varphi}. Consider a Ponomarev-type map $g\colon [0,1]^n \to [0,1]^n$, $g\in W^{1,p}([0,1]^n,\rn)$ which maps a Cantor-set $\mathcal{C}_A$ of measure zero to a Cantor-set $\mathcal{C}_{B}=g(\mathcal{C}_{A})$ of positive measure, and which is identical on $\partial([0,1]^n)$.
    Such a map can be found by the standard construction, see \cite[Chapter~4.3]{HK} with 
    $$
    a_k= \frac{1}{k^{\alpha}}\text{ and }
    b_k=1+\frac{1}{k^{\alpha n}}
    \text{ where }0<\alpha < \min\left\{\frac{n}{p}, \frac{n}{(n-1)\beta}\right\}.
    $$ 
    Referring the reader to \cite[Chapter~4.3]{HK} for details, we just notice that on the $k$-th level we have 
    \begin{equation*}\label{Dg_pointwise}
    \begin{aligned}
        & |Dg| \approx \max\left\{\frac{b_k}{a_k}, \frac{b_{k-1} - b_{k}}{a_{k-1} - a_k}\right\} \approx k^{\alpha},
        \qquad
        J_g \approx \left(\frac{b_k}{a_k}\right)^{n-1} \cdot \frac{b_{k-1} - b_{k}}{a_{k-1} - a_k} \approx 1, \\
        & |\cof Dg| \leq C |Dg|^{n-1} \approx k^{\alpha (n-1)}\quad\text{ and }\quad \left|\{k\text{-th level}\}\right| \approx 2^{-kn} \frac{1}{k^{n+1}}.
    \end{aligned} 
    \end{equation*}
    It follows that the map $g$ has finite energy since  
    \begin{equation}\label{Dg_integral}
    \begin{aligned}
        & \int_{(0,1)^n} |Dg(x)|^p \dx \leq C \sum_{k=1}^{\infty} 2^{kn} 2^{-kn}\frac{1}{k^{n+1}} k^{\alpha p} < \infty,\\
        & \int_{(0,1)^n} A(|\cof Dg(x)|) \dx \leq C \int_{(0,1)^n} |\cof Dg(x)|^{\beta} \dx \leq C \sum_{k=1}^{\infty} 2^{kn} 2^{-kn}\frac{1}{k^{n+1}} k^{\alpha \beta (n-1)} <\infty,\\
        & \int_{(0,1)^n} \varphi(J_{g}(x)) \dx <\infty,
        \qquad \text{and} \qquad
        \int_{(0,1)^n} K_g^{\frac{1}{n-1}}(x) \dx \leq \int_{(0,1)^n} |Dg(x)|^{\frac{n}{n-1}} \dx <\infty.
    \end{aligned}
    \end{equation}

    We set $f_1=g$ and we divide the cube $[0,1]^n$ into $m^n$ equal cubes both in the domain and in the target.
    Fix one of those $m^n$ small cubes
    $Q_z:=\{x\in [0,1]^n : \|x-z\|_{\infty} < \frac{1}{2m}\}$
    with a center point
    $z$
    and 
    define $f_m|_{Q_z}\colon Q_z \to Q_z$ as a scaled and translated copy of $g$ 
    $$
    f_m(x):=\frac{1}{m}g\bigl(m(x-z)\bigr)+z. 
    $$
    It is easy to see by change of variables that
    $$
        \int_{(0,1)^n}|Df_m|^p \dx = \int_{(0,1)^n} |Dg|^p \dx
    $$
    and analogously for other integrals in \eqref{Dg_integral}. 
    It follows that $\sup_m \F (f_m) = \F(g) <\infty$ and hence there is a subsequence which converges weakly in $W^{1,n-1}$. Since $f_m\to f:=\operatorname{id}$ pointwise, identity is a weak limit of $f_m$. 

    By construction, $g$ maps the Cantor-set $\mathcal{C}_{A}$ to the Cantor-set $\mathcal{C}_{B}$, where $\mathcal{C}_{A}$ is a set where $g$ fails the Lusin $\N$ condition. 
    Hence, for every $m\in\en$ and for each $Q_{z}$ it holds that
    $$
        \int_{Q_z} J_{f_m}(x) \dx=\frac{1-|\mathcal{C}_B|}{m^n}.
    $$ 
    Therefore,
    $$
        \int_{(0,1)^n} J_{f_m}(x) \dx =1-|\mathcal{C}_B|<1=\int_{(0,1)^n} J_f(x) \dx,
    $$ 
    so the lower semicontinuity fails at least for this quasiconvex functional. 
\end{proof}

\begin{figure}[h]
\begin{center}
\begin{tikzpicture}[scale=0.02]
    \draw[step=64] (0, 0) grid (192, 192);
    \draw (10,74) rectangle (22, 86);
    \draw (10,106) rectangle (22, 118);
    \draw (42,74) rectangle (54, 86);
    \draw (42,106) rectangle (54, 118);
    \draw (12,76) rectangle (14, 78);
    \draw (12,82) rectangle (14, 84);
    \draw (18,76) rectangle (20, 78);
    \draw (18,82) rectangle (20, 84);
    \draw (12,108) rectangle (14, 110);
    \draw (12,114) rectangle (14, 116);
    \draw (18,108) rectangle (20, 110);
    \draw (18,114) rectangle (20, 116);
    \draw (44,76) rectangle (46, 78);
    \draw (44,82) rectangle (46, 84);
    \draw (50,76) rectangle (52, 78);
    \draw (50,82) rectangle (52, 84);
    \draw (44,108) rectangle (46, 110);
    \draw (44,114) rectangle (46, 116);
    \draw (50,108) rectangle (52, 110);
    \draw (50,114) rectangle (52, 116);
    \draw [->] (200,96) -- (244,96);
    \draw[step=64] (256, 0) grid (448, 192);
    \draw (258,66) rectangle (286, 94);
    \draw (258,98) rectangle (286, 126);
    \draw (290,66) rectangle (318, 94);
    \draw (290,98) rectangle (318, 126);
    \draw (260,68) rectangle (270, 78);
    \draw (260,82) rectangle (270, 92);
    \draw (274,68) rectangle (284, 78);
    \draw (274,82) rectangle (284, 92);
    \draw (260,100) rectangle (270, 110);
    \draw (260,114) rectangle (270, 124);
    \draw (274,100) rectangle (284, 110);
    \draw (274,114) rectangle (284, 124);
    \draw (292,68) rectangle (302, 78);
    \draw (292,82) rectangle (302, 92);
    \draw (306,68) rectangle (316, 78);
    \draw (306,82) rectangle (316, 92);
    \draw (292,100) rectangle (302, 110);
    \draw (292,114) rectangle (302, 124);
    \draw (306,100) rectangle (316, 110);
    \draw (306,114) rectangle (316, 124);
\end{tikzpicture}
\end{center}
\caption{$f_3$ and its action on  $Q_{\left(\frac{1}{6}, \frac{1}{2}\right)}$.}
\label{pic:example}
\end{figure}
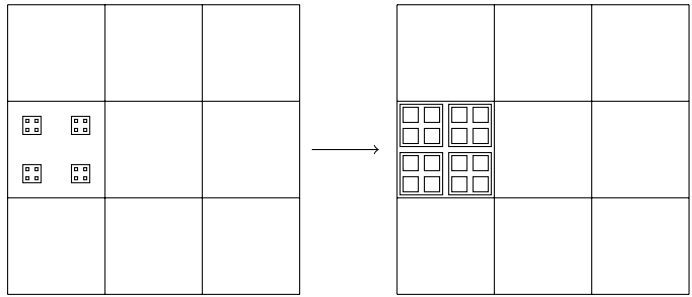

\subsection{$\N$ condition and lower semicontinuity in the context of Theorem \ref{main2}}\label{sec:DHM}


\begin{proof}[Proof of Theorem \ref{main2}]
The fact that the limit $f$ satisfies the $\INV$ condition follows from \cite[Theorem 3.1 a)]{DHM} and Lemma \ref{jfnonzero}.

As in Step 2 of Section~\ref{three}, we know that for every center and almost every radius the corresponding ball satisfies the conditions of Lemma~\ref{lem:fk-ft}. Thus,
    \begin{equation}\label{eq:fk-ft_DHM}
        |f_m(B)\triangle \topi(f,B)|\overset{k\to\infty}{\to} 0.
    \end{equation}

Further, the $\N$ condition of $f$ and a lower semicontinuity of $\mathcal{G}$ follow from Lemmata~\ref{lem:N}--\ref{lem:jacobian_convergence}, provided with \eqref{eq:fk-ft_DHM}. 
    To prove the lower semicontinuity of $\mathcal{G}$, we just note that the function 
    $$
    g(x,y) = x^{n-1} + x^{\frac{n}{n-1}} y^{-\frac{1}{n-1}} + \ff(y)
    $$
    is convex, and as in Lemma~\ref{previous} use the De Giorgi Theorem \cite[Theorem 3.23]{Dac2008} again. Let us note that our functional depends only on $|Df|$ and $\det Df$, but not on $\cof Df$, so we do not need to care about convergence of $(n-1)\times(n-1)$ subdeterminants here. 
    
    Further, from \cite[Corollary 4.2]{DS} we conclude that there exists a subsequence of $f_m^{-1}$ which converges weakly-$*$ in $BV$ to some $h$, in particular $f_m^{-1} \to h$ in $L^{1}_{\loc}$. Injectivity almost everywhere of both $f$ and $h$ is then  obtained by following the lines of proof of Lemma~\ref{inverse}.
\end{proof}

\section{Application to Calculus of Variations}\label{calcvar}

Let $n\geq 3$ and $\Omega$, $\Omega'\subset\rn$ be bounded domains, e.g.\ representing the reference and deformed configurations in nonlinear elasticity. 
Define the energy functional
\begin{equation}\label{elastic_enegry}
    \mathcal{E}(f):=\int_{\Omega} W(Df(x)) \dx,
\end{equation}
where $W \colon \er^{n\times n} \to \er $ is a \textit{polyconvex} function, i.e.,\
$W$ can be expressed as a convex function of the minors of its argument,
satisfying 
\begin{equation}\label{coercivity}
    W(F)\geq 
    \begin{cases}
        C \bigl(|F|^{n-1}+\varphi(\det F)+A(|\cof F|)-1\bigr), & \text{if } \det F >0,\\
        \infty, & \text{if } \det F \leq 0,
        \end{cases}
\end{equation}
for some $C>0$ and for some positive functions $A$ and $\varphi$.
Consider a homeomorphism $f_0$ from $\overline{\Omega}$ onto $\overline{\Omega'}$ such that $\mathcal{E}(f_0)<\infty$, and 
the following sets of admissible functions:
\begin{multline*}
    \mathcal{H}_{f_0}(\Omega,\rn):=\Bigl\{f:\overline{\Omega}\to\rn: f\text{ is a homeomorphism 
    of }\overline{\Omega}\text{ onto }\overline{\Omega'} \text{ satisfying }\\
    \text{ the Lusin } \N\text{ condition, } 
    f=f_0\text{ on }\partial \Omega, 
    \text{ and } \mathcal{E}(f) \leq \mathcal{E}(f_0)\Bigr\}
\end{multline*}
and 
\begin{multline*}
    \overline{\mathcal{H}}^{w}_{f_0}(\Omega,\rn):=\Bigl\{f:\Omega\to\rn:\text{ there are }f_m\in \mathcal{H}_{f_0}(\Omega,\rn)\text{ with } 
    \\
    f_m\rightharpoonup f\text{ weakly in }W^{1,n-1}(\Omega,\rn)\Bigr\}.
\end{multline*}

Note that $\overline{\mathcal{H}}^{w}_{f_0}(\Omega,\rn)$ is weakly (sequentially) closed and hence it is a suitable set of mappings for variational approach: 

\begin{prop}\label{prop:H_closure}
    Let $g_m\in \overline{\mathcal{H}}^{w}_{f_0}(\Omega,\rn)$ and assume that $g_m\rightharpoonup g$ weakly in $W^{1,n-1}(\Omega,\rn)$. 
    Then $g\in \overline{\mathcal{H}}^{w}_{f_0}(\Omega,\rn)$.
    In particular, there exists a sequence $f_{m}\in \mathcal{H}_{f_0}(\Omega,\rn)$ such that $f_m\rightharpoonup g$ weakly in $W^{1,n-1}(\Omega,\rn)$,
    $g=f_0$ on $\partial \Omega$, and 
    $\sup_{m} \mathcal{E}(f_m) \leq \mathcal{E}(f_0)<\infty$. 
\end{prop}

\begin{proof}
    Since $W^{1,n-1}$ is reflexive and separable we can find $\{L_i\}_{i\in\en}\subset (W^{1,n-1})^*$ which is dense. 
    We can assume (passing to a subsequence) that
    $$
        |L_i(g_m-g)|<\frac{1}{k} \text{ for every }i\in\{1,\hdots,m\}. 
    $$
    For every $g_m$ we can find a sequence in $\mathcal{H}_{f_0}(\Omega,\rn)$ which converges weakly and thus we can fix $f_m\in \mathcal{H}_{f_0}(\Omega,\rn)$ such that
    $$
        |L_i(g_m-f_m)|<\frac{1}{m} \text{ for every }i\in\{1,\hdots,m\}. 
    $$
    It follows that for every $i\in\en$ we have
    $$
        \lim_{m\to\infty} L_i(f_m)=L_i(g). 
    $$
    Since $f_m(x)\in\Omega'\subset B(0,R)$ for all $x\in\Omega$,
    $\|Df_m\|_{L^{n-1}}\leq\mathcal{E}(f_m)\leq \mathcal{E}(f_0)$ result in $\|f_m\|_{W^{1,n-1}}\leq M$ for all $m$ and some constant $M>0$, so we easily obtain that 
    $$
        \lim_{m\to\infty} L(f_m)=L(g)\text{ for every } L\in (W^{1,n-1})^*.
    $$

    Note further that the set $f_0+W_0^{1,n-1}(\Omega,\rn)$ is closed and convex and thus weakly closed,
    therefore, $g=f_0$ on $\partial \Omega$.
\end{proof}

\begin{prop}\label{prop:jacobian_convergence_6}
    Let $g$, $g_m\in \overline{\mathcal{H}}^{w}_{f_0}(\Omega,\rn)$ and assume that $g_m\rightharpoonup g$ weakly in $W^{1,n-1}(\Omega,\rn)$. 
    Then (up to subsequence) $J_{g_m} \rightharpoonup J_{g}$ weakly in $L^1(\Omega)$.
\end{prop}
\begin{proof}
    Let us first prove that
    for every ball $B$, such that $f$ satisfies the $\INV$ condition and \eqref{eq:fk-ft} for $B$, 
    \begin{equation} \label{topi_convergence}
        \lim_{m\to\infty}|\topi (g_m,B)|= |\topi (g,B)|. 
    \end{equation}
    Fix such a ball $B$ for all $m\in\en$ and $\varepsilon >0$.
    In view of Lemma~\ref{lem:fk-ft}, for any $m\in\en$ there exists a sequence $f_{m,k}\in \mathcal{H}_{f_0}(\Omega,\rn)$ such that 
    $$
        |f_{m,k}(B) \triangle \topi (g_m,B)| \leq \frac{\varepsilon}{2}.
    $$
    Using the diagonal procedure as in Proposition~\ref{prop:H_closure}, we find a sequence $f_{m,k(m)}\in \mathcal{H}_{f_0}(\Omega,\rn)$ with $f_{m,k(m)}\rightharpoonup g$ weakly in $W^{1,n-1}$. 
    Again by Lemma~\ref{lem:fk-ft} it holds that
    $$
        |f_{m,k(m)}(B) \triangle \topi (g,B)| \leq \frac{\varepsilon}{2}
    $$
    for $m$ big enough.
    Combining these two inequalities, we obtain \eqref{topi_convergence}.

    Now the proof follows proof of Lemma~\ref{lem:jacobian_convergence},
    since 
    for every $a\in\Omega$ there is $r_a>0$ such that for $\mathcal{H}^1$-a.e.~$r\in (0,r_a)$ the mapping $g$ satisfies \eqref{eq:fk-ft} and $\INV$ for $B(a,r )$.
\end{proof}

\begin{thm}\label{CalcVar-1} 
    Let $n\geq 3$ and $\Omega$, $\Omega'\subset\rn$ be bounded domains,
    let also $W\colon \er^{n\times n} \to \er$ be a polyconvex function satisfying 
    \eqref{coercivity}
    for some functions $A$ and $\varphi$ satisfying \eqref{varphi}--\eqref{A} and a constant $C>0$.
    Assume further that $f_0$ is a homeomorphism from $\overline{\Omega}$ onto $\overline{\Omega'}$ such that $\mathcal{E}(f_0)<\infty$, where $\mathcal{E}$ is the energy defined by~\eqref{elastic_enegry}.  
    Then there exists $f\in \overline{\mathcal{H}}^{w}_{f_0}(\Omega,\rn)$ such that
    $$
        \mathcal{E}(f)=\inf\bigl\{\mathcal{E}(h):\ h\in \overline{\mathcal{H}}^{w}_{f_0}(\Omega,\rn)\bigr\}. 
    $$
    Moreover, $f$ satisfies the $\INV$ condition and the Lusin $\N$ condition. 
\end{thm}
\begin{proof}
Let $f_{m}$ be a minimizing sequence for $\mathcal{E}$, 
then $f_m$ form a bounded sequence in $W^{1,n-1}$, 
and hence using Proposition \ref{prop:H_closure} there is $f\in \overline{\mathcal{H}}^{w}_{f_0}(\Omega,\rn)$ such that (up to a subsequence)
$f_m\rightharpoonup f$ weakly in $W^{1,n-1}$.
Provided with Proposition~\ref{prop:jacobian_convergence_6}, we obtain that $\mathcal{E}$ is lower semicontinuous in $\overline{\mathcal{H}}^{w}_{f_0}(\Omega,\rn)$ following the proof of Lemma~\ref{previous}.

Proposition~\ref{prop:H_closure} and
Theorem \ref{main} imply thus that $f$ satisfies the $\INV$ and the $\N$ conditions and also that 
$$
    \mathcal{E}(f)\leq\liminf_{m\to\infty}\mathcal{E}(f_m)=
    \lim_{m\to\infty}\mathcal{E}(f_m)=\inf\bigl\{\mathcal{E}(h):\ h\in \overline{\mathcal{H}}^{w}_{f_0}(\Omega,\rn)\bigr\}\leq \mathcal{E}(f). 
$$ 
\end{proof}

\begin{remark}
    Let us note that it is not clear if the two following infima 
    $$
        \inf\bigl\{\mathcal{E}(h):\ h\in \mathcal{H}_{f_0}(\Omega,\rn)\bigr\}
        \quad\text{ and }\quad  \inf\bigl\{\mathcal{E}(h):\ h\in \overline{\mathcal{H}}^{w}_{f_0}(\Omega,\rn)\bigr\},
    $$
    are equal or not since the space ${\mathcal{H}}_{f_0}(\Omega,\rn)$ is not compact.  
\end{remark}

Analogously  we can use the results of \cite{DHM} and Section~\ref{sec:DHM} to obtain the following theorem.
\begin{thm}\label{CalcVar-2} 
    Let $n\geq 3$ and $\Omega$, $\Omega'\subset\rn$ be Lipschitz domains,
    let also $W\colon \er^{n\times n} \to \er$ be a polyconvex function satisfying 
    \begin{equation*}
        W(F)\geq 
        \begin{cases}
            C \Bigl(|F|^{n-1}+\ff(\det F)+\Bigl(\frac{|F|^n}{\det F}\Bigr)^{\frac{1}{n-1}}-1\Bigr), & \text{if } \det F >0,\\
            \infty, & \text{if } \det F \leq 0,
        \end{cases} 
    \end{equation*}
    for some function $\varphi$ satisfying \eqref{varphi}--\eqref{varphi2} and \eqref{varphi3}, and a constant $C>0$. We assume that $W$ may be represented as a convex function of subdeterminants of order strictly less than $n-1$ and of $\det F$, i.e., it is not a function of subdeterminants of order $n-1$. 
    Assume further that $f_0$ is a homeomorphism from $\overline{\Omega}$ onto $\overline{\Omega'}$ such that $\mathcal{E}(f_0)<\infty$, where $\mathcal{E}$ is the energy defined by~\eqref{elastic_enegry}.  
    Then there exists $f\in \overline{\mathcal{H}}^{w}_{f_0}(\Omega,\rn)$ such that
    $$
        \mathcal{E}(f)=\inf\bigl\{\mathcal{E}(h):\ h\in \overline{\mathcal{H}}^{w}_{f_0}(\Omega,\rn)\bigr\}. 
    $$
    Moreover, $f$ satisfies the $\INV$ condition and the Lusin $\N$ condition. 
\end{thm}
\begin{proof}
The proof is analogous to the proof of Theorem \ref{CalcVar-1}. 
The only difference is that in the proof of lower semicontinuity we do not have \eqref{A} and therefore we cannot prove weak convergence of $\cof Df_m$ as in the proof of Lemma \ref{previous}. However, we do not need this as our $W$ ``does not depend'' on $(n-1)\times(n-1)$ subdeterminants. 
\end{proof}


\begin{thebibliography}{10}

\bibitem{AFP}
L.~Ambrosio, N.~Fusco, and D.~Pallara, {\em Functions of bounded variation and
  free discontinuity problems}.
\newblock Oxford Mathematical Monographs, The Clarendon Press, Oxford
  University Press, New York, 2000.

\bibitem{Ball}
J.~M. Ball, ``Convexity conditions and existence theorems in nonlinear
  elasticity,'' {\em Arch. Ration. Mech. Anal.}, vol.~63, pp.~337--403, 1977.

\bibitem{BM}
J.~M. Ball and F.~Murat, ``{$W^{1,p}$}-quasiconvexity and variational problems
  for multiple integrals,'' {\em J. Funct. Anal.}, vol.~58, no.~3,
  pp.~225--253, 1984.

\bibitem{BHMC}
M.~Barchiesi, D.~Henao, and C.~Mora-Corral, ``Local invertibility in {Sobolev}
  spaces with applications to nematic elastomers and magnetoelasticity,'' {\em
  Arch. Ration. Mech. Anal.}, vol.~224, pp.~743--816, 2017.


\bibitem{BHMCR}
M.~Barchiesi, D.~Henao, C.~Mora-Corral, and R.~Rodiac, ``Harmonic dipoles and
  the relaxation of the neo-hookean energy in 3d elasticity,'' {\em Arch. Ration. Mech. Anal.}, vol.~247, no.~70, 2023.

\bibitem{BHMCR2}
M.~Barchiesi, D.~Henao, C.~Mora-Corral, and R.~Rodiac, ``On the lack of
  compactness problem in the axisymmetric neo-hookean model,'' 2021.

\bibitem{BHM}
O.~Bouchala, S.~Hencl, and A.~Molchanova, ``Injectivity almost everywhere for
  weak limits of {S}obolev homeomorphisms,'' {\em J. Funct. Anal.}, vol.~279,
  no.~7, pp.~108658, 32, 2020.
	
\bibitem{BN}
H.~Brezis and L.~Nirenberg, ``Degree theory and {BMO}. {I}. {C}ompact manifolds
  without boundaries,'' {\em Selecta Math. (N.S.)}, vol.~1, no.~2,
  pp.~197--263, 1995.
	
\bibitem{CM}
P.~Celada and G.~Dal~Maso, ``Further remarks on the lower semicontinuity of
  polyconvex integrals,'' {\em Ann. Inst. H. Poincar\'{e} C Anal. Non
  Lin\'{e}aire}, vol.~11, no.~6, pp.~661--691, 1994.
	
\bibitem{Ciar1988}
P.~G. Ciarlet, {\em Mathematical Elasticity, Vol. I : Three-Dimensional
  Elasticity, Series ``Studies in Mathematics and its Applications''}.
\newblock 1988.

\bibitem{CN}
P.~G. Ciarlet and J.~Ne\v{c}as, ``Injectivity and self-contact in nonlinear
  elasticity,'' {\em Arch. Ration. Mech. Anal.}, vol.~97, no.~3, pp.~171--188,
  1987.

\bibitem{CDL}
S.~Conti and C.~De~Lellis, ``Some remarks on the theory of elasticity for
  compressible {Neohookean} materials,'' {\em Ann. Sc. Norm. Super. Pisa Cl.
  Sci. (5)}, vol.~2, pp.~521--549, 2003.
	
\bibitem{CHM}
M.~Cs\"{o}rnyei, S.~Hencl, and J.~Mal\'{y}, ``Homeomorphisms in the {S}obolev
  space {$W^{1,n-1}$},'' {\em J. Reine Angew. Math.}, vol.~644, pp.~221--235,
  2010.
	
\bibitem{Dac2008}
B.~Dacorogna, {\em Direct methods in the calculus of variations}, vol.~78 of
  {\em Applied Mathematical Sciences}.
\newblock Springer, New York, second~ed., 2008.
	
\bibitem{DMS}
G.~Dal~Maso and C.~Sbordone, ``Weak lower semicontinuity of polyconvex
  integrals: a borderline case,'' {\em Math. Z.}, vol.~218, no.~4,
  pp.~603--609, 1995.	

\bibitem{DPP}
G.~De~Philippis and A.~Pratelli, ``The closure of planar diffeomorphisms in
  {S}obolev spaces,'' {\em Ann. Inst. H. Poincar\'{e} C Anal. Non
  Lin\'{e}aire}, vol.~37, no.~1, pp.~181--224, 2020.

\bibitem{DHM}
A.~Dole\v{z}alov\'a, S.~Hencl, and J.~Mal\'y, ``Weak limit of homeomorphisms in
  {$W^{1,n-1}$} and {$\INV$} condition,'' {\em Arch. Ration. Mech. Anal.}, vol.~247, no.~80, 2023.
	
\bibitem{DS}
L.~D'Onofrio and R.~Schiattarella, ``On the total variations for the inverse of
  a {BV}-homeomorphism,'' {\em Adv. Calc. Var.}, vol.~6, no.~3, pp.~321--338,
  2013.
	
\bibitem{EG}
L.C.~Evans and R.F.~Gariepy, {\em Measure theory and fine properties of functions}, 
  {\em Studies in Advanced Mathematics}.
\newblock CRC Press, Boca Raton, FL, 1992. viii+268 pp.

\bibitem{Fe}
H.~Federer, {\em Geometric measure theory}, 
  {\em Die Grundlehren der mathematischen Wissenschaften,
Band 153}.
\newblock Springer-Verlag, New York, 1969
(Second edition 1996).
 	
\bibitem{FusMosSbo2008}
N.~Fusco, G.~Moscariello, and C.~Sbordone, ``The limit of {$W^{1,1}$}
  homeomorphisms with finite distortion,'' {\em Calc. Var. Partial Differential
  Equations}, vol.~33, no.~3, pp.~377--390, 2008.


\bibitem{GiaModSou}
M.~Giaquinta, G.~Modica, and J.~Sou\v{c}ek,
``Cartesian currents in the calculus of variations. {I},''
{\em Ergebnisse der Mathematik und ihrer Grenzgebiete. 3. Folge. A
Series of Modern Surveys in Mathematics}, 37.
Springer-Verlag, Berlin, 1998.

	
\bibitem{HH}
P.~Harjulehto and P.~H\"{a}st\"{o}, {\em Orlicz spaces and generalized {O}rlicz
  spaces}, vol.~2236 of {\em Lecture Notes in Mathematics}.
\newblock Springer, Cham, 2019.


\bibitem{HMC}
D.~Henao and C.~Mora-Corral, ``Lusin's condition and the distributional
  determinant for deformations with finite energy,'' {\em Adv. Calc. Var.},
  vol.~5, no.~4, pp.~355--409, 2012.

\bibitem{HeMo11}
D.~Henao and C.~Mora-Corral, ``Fracture surfaces and the regularity of inverses
  for {BV} deformations,'' {\em Arch. Ration. Mech. Anal.}, vol.~201, no.~2,
  pp.~575--629, 2011.
  
\bibitem{HMO}
D.~Henao, C.~Mora-Corral and M.~Oliva, ``Global invertibility of Sobolev maps,'' {\em Adv. Calc. Var.}, vol.~14, no.~2,
  pp.~207--230, 2019.
	
\bibitem{HK}
S.~Hencl and P.~Koskela, {\em Lectures on mappings of finite distortion}.
\newblock Lecture Notes in Mathematics, Vol. 2096, Springer International
  Publishing, 2014.
	
\bibitem{HM}
S.~Hencl and J.~Mal\'{y}, ``Jacobians of {S}obolev homeomorphisms,'' {\em Calc.
  Var. Partial Differential Equations}, vol.~38, no.~1-2, pp.~233--242, 2010.

\bibitem{IO}
T.~Iwaniec and J.~Onninen, ``Monotone {S}obolev mappings of planar domains and
  surfaces,'' {\em Arch. Ration. Mech. Anal.}, vol.~219, no.~1, pp.~159--181,
  2016.

\bibitem{IO2}
T.~Iwaniec and J.~Onninen, ``Limits of {S}obolev homeomorphisms,'' {\em J. Eur.
  Math. Soc. (JEMS)}, vol.~19, no.~2, pp.~473--505, 2017.
	
\bibitem{L}
G.~Leoni, {\em A first course in {S}obolev spaces}, vol.~181 of {\em Graduate
  Studies in Mathematics}.
\newblock American Mathematical Society, Providence, RI, second~ed., 2017.
	
\bibitem{M}
J.~Mal\'{y}, ``Weak lower semicontinuity of polyconvex integrals,'' {\em Proc.
  Roy. Soc. Edinburgh Sect. A}, vol.~123, no.~4, pp.~681--691, 1993.

\bibitem{MM}
J.~Mal\'y and O.~Martio, ``Lusin's condition $\N$ and mappings of the class $W^{1,n}$,'' {\em J. Reine Angew. Math.}, vol.~458,
  pp.~19--36, 1995.


\bibitem{MS}
S.~M\"uller and S.~Spector, ``An existence theory for nonlinear elasticity that
  allows for cavitation,'' {\em Arch. Ration. Mech. Anal.}, vol.~131, no.~1,
  pp.~1--66, 1995.

\bibitem{MST}
S.~M\"uller, S.~Spector, and Q.~Tang, ``Invertibility and a topological
  property of {S}obolev maps,'' {\em SIAM J. Math. Anal.}, vol.~27,
  pp.~959--976, 1996.
  
\bibitem{Ri}
F.~Rindler, {\em Calculus of variations}, {\em Universitext}. 
\newblock  Springer, Cham, 2018, 444pp.


\bibitem{SciStr2022}
G.~Scilla and B.~Stroffolini,
``Invertibility of {O}rlicz-{S}obolev maps,''
in {\em Research in mathematics of materials science,
Assoc. Women Math. Ser.},
vol.~31,
pp.~297--317, 2022.


\bibitem{STY}
J.~Spector, Q.~Tang and B.~S.~Yan, ``On a new class of elastic deformations
not allowing for cavitation,''
  {\em Ann. Inst. H. Poincaré Anal. Non Linéaire}, vol.~11, no.~2, pp.~217--243, 1994.

\bibitem{SwaZie2002}
D.~Swanson and W.~P. Ziemer, ``A topological aspect of {S}obolev mappings,''
  {\em Calc. Var. Partial Differential Equations}, vol.~14, no.~1, pp.~69--84,
  2002.

\bibitem{SwaZie2004}
D.~Swanson and W.~P. Ziemer, ``The image of a weakly differentiable mapping,''
  {\em SIAM J. Math. Anal.}, vol.~35, no.~5, pp.~1099--1109, 2004.

\bibitem{T}
Q.~Tang, ``Almost-everywhere injectivity in nonlinear elasticity,'' {\em Proc.
  Roy. Soc. Edinburgh Sect. A}, vol.~109, no.~1--2, pp.~79--95, 1988.


\end{thebibliography}
\end{document}